\documentclass[11pt,reqno]{amsart}
\usepackage{amssymb}
\usepackage{amsmath}
\usepackage{amsthm}
\usepackage{verbatim}
\usepackage[top=1in, bottom=1in, left=1in, right=1in]{geometry}
\usepackage[stable]{footmisc}
\usepackage{indentfirst}

\numberwithin{equation}{section}

\input{xy}
\xyoption{all}

\theoremstyle{definition}
\newtheorem*{definition*}{Definition}
\newtheorem{definition}{Definition}

\theoremstyle{remark}
\newtheorem{remark}{Remark}

\newtheorem*{note*}{Note}

\newtheorem*{example*}{Example}

\newtheorem*{question*}{Question}
\newtheorem*{blank*}{}

\theoremstyle{plain}
\newtheorem{theorem}{Theorem}
\newtheorem*{theorem*}{Theorem}
\newtheorem{corollary}{Corollary}
\newtheorem*{corollary*}{Corollary}
\newtheorem{lemma}{Lemma}
\newtheorem*{lemma*}{Lemma}
\newtheorem{proposition}{Proposition}
\newtheorem*{proposition*}{Proposition}

\newtheorem*{conjecture*}{Conjecture}

\newtheorem*{assump*}{Assumption}

\newcommand{\ba}{\mathbb A}

\newcommand{\bc}{\mathbb C}

\newcommand{\bk}{\mathbb K}

\newcommand{\bn}{\mathbb N}

\newcommand{\bp}{\mathbb P}
\newcommand{\bq}{\mathbb Q}
\newcommand{\br}{\mathbb R}

\newcommand{\bz}{\mathbb Z}

\newcommand{\ma}{\mathcal A}

\newcommand{\ml}{\mathcal L}

\newcommand{\mpp}{\mathcal P} 

\newcommand{\ms}{\mathcal S}

\newcommand{\muu}{\mathcal U} 

\newcommand{\mz}{\mathcal Z}

\newcommand{\fa}{\mathfrak a}
\newcommand{\fb}{\mathfrak b}

\newcommand{\fg}{\mathfrak g}
\newcommand{\fh}{\mathfrak h}

\newcommand{\fk}{\mathfrak k}

\newcommand{\fm}{\mathfrak m}

\newcommand{\fp}{\mathfrak p}

\newcommand{\fs}{\mathfrak s}
\newcommand{\ft}{\mathfrak t}
\newcommand{\fu}{\mathfrak u}

\newcommand{\fz}{\mathfrak z}

\newcommand{\fD}{\mathfrak D}

\newcommand{\fS}{\mathfrak S}

\newcommand{\bbq}{{\bar{\bq}}}
\newcommand{\bqab}{\bq^{ab}}

\newcommand{\rar}{\rightarrow}

\newcommand{\lrar}{\longrightarrow}

\newcommand{\hrar}{\hookrightarrow}

\newcommand{\dlim}{\varinjlim}
\newcommand{\ilim}{\varprojlim}

\newcommand{\wtilde}{\widetilde}

\newcommand{\pprime}{{\prime\prime}}

\newcommand{\smalltwomatrix[5]}{
 \left( \begin{smallmatrix}
 {#2} & {#3}\\
 {#4} & {#5}
 \end{smallmatrix}\right)}

 \DeclareMathOperator\Hom{Hom}
 
\DeclareMathOperator\Res{Res} 
\DeclareMathOperator\Ad{Ad} 
\DeclareMathOperator\Ker{Ker}

 \DeclareMathOperator\CH{CH}
\DeclareMathOperator\Gal{Gal} 
 
\DeclareMathOperator\re{Re}  
 \DeclareMathOperator\Lie{Lie}
 \DeclareMathOperator\diag{diag}
 
\DeclareMathOperator\Tr{Tr} 
 \DeclareMathOperator\sgn{sgn}

\newcommand{\leftup}[2]{{^{#1}}\mspace{-2.5mu}#2}

\newcommand{\tran}[1]{\leftup{t}{#1}}

\newcommand{\ppder}[2]{\frac{\partial #1}{\partial #2}}

\newcommand{\isoto}{\stackrel{\sim}{\lrar}}

\input{xy}
\xyoption{all}

\DeclareMathOperator\SC{SC}

\newcommand{\mkf}{{M_{K_f}}}

\newcommand{\tmkf}{\wtilde{M}_{K_f}}

\newcommand{\HH}{\mathrm{H}}

\newcommand{\Ta}{\mathrm{Ta}}

\newcommand{\cl}{\mathrm{cl}}
\newcommand{\Sp}{\mathrm{Sp}}
\newcommand{\GSp}{\mathrm{GSp}}
\newcommand{\PGSp}{\mathrm{PGSp}}
\newcommand{\GL}{\mathrm{GL}}
\newcommand{\SL}{\mathrm{SL}}
\newcommand{\PGL}{\mathrm{PGL}}
\newcommand{\SO}{\mathrm{SO}}
\newcommand{\OO}{\mathrm{O}}
\newcommand{\GSpin}{\mathrm{GSpin}}
\newcommand{\Wd}{\mathrm{Wd}}

\newcommand{\disc}{\mathrm{disc}}
\begin{document}

\title{Tate Classes on Siegel 3-folds }
\author[ Yannan Qiu]{Yannan Qiu}
\email{yannan.qiu@gmail.com}
\maketitle
\begin{abstract} This article proves that Tate classes on Siegel modular 3-folds are spanned by the images of Hilbert modular surfaces at degree 2 and by the images of Shimura curves at degree 4. The proof involves a careful study of the period pairing between degree-4 rapidly decreasing cohomological forms and Hilbert modular surfaces.
\end{abstract}

\setcounter{tocdepth}{1}
\tableofcontents


\noindent [\emph{An earlier version of this paper has circulated as a preprint for some years. Due to a few recent requests of the paper made to the author, the current version is prepared with some minor notations changed.}]

Let $F$ be a number field, $X$ be a $n$-dimensional smooth projective $F$-variety and set $X_\bbq=X\otimes_F \bbq$. With respect to the $\ell$-adic cycle map
\begin{equation}\label{intro_cl}
\cl_{et}:\CH^i(X_\bbq)\lrar \mathrm{H}^{2i}(X_\bbq, \bz_\ell(i)),
\end{equation}
Tate Conjecture predicts that for any finite extension $E/F$, the space $\mathrm{H}^{2i}(X_\bbq, \bq_\ell(i))^{\Gal(\bbq/E)}$ is spanned by the images of codimension-$i$ $E$-cycles. An accompanying question is to construct sufficient concrete $E$-cycles so that their images span $\mathrm{H}^{2i}(X_\bbq, \bq_\ell(i))^{\Gal(\bbq/E)}$.

When $X$ is a Shimura variety associated to a reductive $\bq$-group $G$ with the reflex field $F$, $X$ is defined over $F$. Each reductive $\bq$-subgroup of $G$ yields a Shimura subvariety, whose connected components are defined over finite abelian extensions of $F$. These connected components and their Hecke translates are called special cycles in $X$. Ramakrishnan raised a general question in \cite{Rama90}: are special cycles enough to generate the $\Gal(\bbq/F^{\mathrm{ab}})$-invariant subspace of the intersection cohomology $\mathrm{IH}^{2i}(X_\bbq,\bq_\ell(i))$?

The original example motivating Ramakrishnan's question is Hilbert modular surfaces $S$, for which the reflex field is $F=\bq$. Harder, Langlands and Rapport \cite{hlr1986} proved the Tate conjecture for $\mathrm{H}^2(S_{\bbq},\bq_\ell(1))$. They also discovered that modular curves are enough to generate the $\Gal(\bbq/\bqab)$-invariant subspace of the interior cohomology $\mathrm{H}^2_!(S_\bbq,\bq_\ell(1))$.

This paper provides Siegel 3-folds as a second affirmative example in this direction. Let $\ba:=\ba_\bq$ be the ring of adels over $\bq$ and $\ba_f$ be the subring of finite adels. For a neat compact open subgroup $K_f$ of $\GSp_4(\ba_f)$, let $\mkf$ denote the Siegel $3$-fold of level $K_f$.  $\mkf$ is defined over $\bq$ and we let $\tmkf$ be a smooth toroidal compactification which is defined over $\bq$ and whose boundary divisors have normal crossings. Weissauer proved the Tate conjecture for $\mathrm{H}^2(\tmkf\otimes_\bq \bbq,\bq_\ell(1))$ \cite[Thm. 9.4]{Weiss88} and showed that $\Gal(\bbq/\bq^{ab})$ acts trivially on $\CH^1(\mkf\otimes_\bq \bbq)\otimes_\bz \bq$ \cite[Thm. 2]{Weiss92}. Based on his work, we prove that Tate classes of  degree $2$ and $4$ are spanned by the images of special cycles and cycles on the boundary.

Let $\SC^i(\mkf)$ denote the subgroup of $\CH^i(\mkf\otimes_\bq \bbq)$ consisting of special cycles of codimension $i$ (see Sect. ~\ref{cycle}). When $i=1$, special cycles are Hecke translates of Hilbert modular surfaces in $\mkf$; when $i=2$, special cycles are Hecke translates of Shimura curves in $\mkf$. For a cycle $Z$ in $\mkf\otimes_\bq \bbq$, let $\overline{Z}$ denote its Zariski closure in $\tmkf\otimes_\bq \bbq$.

With respect to the cycle map (\ref{intro_cl}), let $\CH^i_0(X_\bbq)$ denote the kernel and $\overline{\CH}^i(X_\bbq)$ be the quotient group $\CH^i(X_\bbq)/\CH^i_0(X_\bbq)$. Let $\Ta^i(X_\bbq)$ be the union of $\mathrm{H}^{2i}(X_\bbq, \bq_\ell(i))^{\Gal(\bbq/E)}$ for all number fields $E\supset F$; it is the space of all degree-$2i$ Tate classes on $X_\bbq$.


Our main theorem is below.
\begin{theorem}\label{main}
Let $\{B_i\}_{i=1}^m$ be the boundary divisors of $\tmkf\otimes_\bq \bbq$.
\begin{itemize}
\item[(1)] $\CH^1(\mkf\otimes_\bq \bbq)\otimes_\bz \bq=\SC^1(\mkf)\otimes_\bz \bq$.
\item[(2)] $\Ta^1(\tmkf)$ is spanned by the images of $[\overline{Z}]$ and $[B_i]$, with $[Z]\in \SC^1(\mkf)$ and $1\leq i\leq m$.

\item[(3)] $\Ta^2(\tmkf)$ is spanned by the images of $[\overline{Z}]$, $[B_i\cdot B_j]$ and $[\overline{Z}\cdot B_i]$, with $[Z]\in \SC^2(\mkf)$ and $1\leq i,j\leq m$.

\item[(4)] $\overline{\CH}^2(\mkf\otimes_\bq \bbq)\otimes_\bz \bq$ is spanned by the image of $\SC^2(\mkf)$.
\end{itemize}
\end{theorem}
\begin{remark}\label{thm_com}
When the variety $X$ in (\ref{intro_cl}) is smooth, let $\HH^\ast(X,R)$ denote the singular cohomology of the complex manifold $X(\bc)$ with coefficients in an abelian group $R$. There is a cycle map
\[
\cl: \CH^i(X_\bbq)\rar \HH^{2i}(X,\bz)
\]
The Artin comparison theorem of etale cohomology and singular cohomology gives a canonical isomorphism $\HH^\ast(X_\bbq, \bz_\ell)\cong \HH^\ast(X,\bz)\otimes \bz_\ell$ that is compatible with $\cl_{et}$ and $\cl$, whence $\CH^i_0(X_\bbq)$ is also the kernel of $\cl$.
\end{remark}

\begin{remark}\label{thm_com2}
$\CH^1_0(\mkf\otimes_\bq \bbq)\otimes \bq$ vanishes because of the well-known fact $\HH^1(\mkf,\bc)=0$ (See Lemma ~\ref{cyclelemma}). $\CH^2_0(\mkf\otimes_\bq \bbq)$ is more difficult to study and is related to the intermediate Jacobian. It is a very interesting problem to characterize $\CH^2_0(\mkf\otimes_\bq \bbq)\cap \SC^2(\mkf)$.
\end{remark}

In Theorem ~\ref{main}, there is (1)$\Longrightarrow$(2)$\Longrightarrow$(3)$\Longrightarrow$(4). By Remark ~\ref{thm_com2}, the following cycle map with coefficients in $\bq$ is injective,
\[
\cl_{\mkf}: \CH^1(\mkf\otimes_\bq \bbq)\otimes_\bz \bq\rar \HH^2(\mkf,\bq).
\]
To prove (1),  it suffices that the images of $\CH^1(\mkf\otimes_\bq \bbq)$ and $\SC^1(\mkf)$ span the same subspace in $\HH^2(\mkf,\bq)$, or equivalently, in $\HH^2(\mkf,\bc)$.

We choose to handle all $K_f$ simultaneously. Write $M:=\ilim_{K_f} M_{K_f}$,  $\SC^i(M):=\dlim_{K_f}\SC^i(\mkf)$, $\CH^i(M\otimes_\bq \bbq):=\dlim_{K_f} \CH^i(\mkf\otimes_\bq \bbq)$, and $\mathrm{H}^{2i}(M,\bc):=\dlim \HH^{2i}(\mkf,\bc)$. The according cycle map
\begin{align}\label{intro_clc}
&\mathrm{cl}_M:\CH^1(M\otimes_\bq \bbq)\otimes_\bz \bc\rar \mathrm{H}^{2}(M,\bc).
\end{align}
is injective, $\GSp_4(\ba_f)$-equivariant and factors through $\HH^{1,1}(M,\bc)$.

\begin{theorem}\label{picardgroup}
There is an isomorphism
\[
 \mathrm{cl}_M|_{\SC^1(M)}:\SC^1(M)\otimes_\bq \bc \isoto \mathrm{H}^{1,1}(M,\bc).
\]
\end{theorem}

Theorem ~\ref{picardgroup} implies $\SC^1(M)\otimes_\bz \bc=\CH^1(M\otimes_\bq \bbq)\otimes_\bz \bc$ and is sufficient to deduce Part (1) of Theorem ~\ref{main}. The main body of this paper is to prove Theorem ~\ref{picardgroup} by using the period pairing between cycles and cohomological differential forms. We sketch the main steps of the proof.
\begin{itemize}
\item[Step 1.] $\HH^2(M,\bc)$ is isomorphic to the $(\fg,K_\infty)$-cohomology of the discrete spectrum\\
$L^2_{\disc}(\GSp_4(\bq)\br^\times_+\backslash \GSp_4(\ba))$ and hence is completely reducible as an admissible $\GSp_4(\ba_f)$-module. (See Sect. ~\ref{algebra} and ~\ref{3fold} for the definition of $\fg$ and $K_\infty$.) For an irreducible admissible unitary representation $\pi_f$ of $\mathrm{GSp}_4(\ba_f)$, let $\SC^1(\pi_f)$, $\HH^2(\pi_f)$ and $\mathrm{H}^{1,1}(\pi_f)$ be the $\pi_f$-isotypic component of $\SC^1(M)\otimes_\bz \bc$, $\HH^2(M,\bc)$ and $\mathrm{H}^{1,1}(M,\bc)$ respectively. It reduces to check that the cycle map $\cl(\pi_f): \SC^1(\pi_f)\hrar\mathrm{H}^{1,1}(\pi_f)$ on each isotypic component is an isomorphism. We show that when $\HH^{1,1}(\pi_f)$ is nonzero,  $\HH^{1,1}(\pi_f)=\HH^2(\pi_f)$ is irreducible. So it suffices that $\SC^1(\pi_f)$ is nonzero when $\pi_f$ occurs in $\HH^{1,1}(M,\bc)$.

\item[Step 2.] The map $\cl$ is defined with respect to the Poincar\'{e} duality,
\begin{equation}\label{intro_pd}
\mathrm{H}^2(M,\bc)\times \mathrm{H}^4_c(M,\bc)\lrar \bc.
\end{equation}
Here the singular cohomology $\mathrm{H}^\ast(M,\bc)$ and $\mathrm{H}^\ast_c(M,\bc)$ can be identified with the de Rham cohomology and de Rham cohomology with compact support. Also, the pairing in (\ref{intro_pd}) respects the action of $\GSp_4(\ba_f)$ and hence is a direct sum of the perfect pairings on the isotypic components,
\[
\HH^2(\pi_f)\times \HH^4_c(\pi_f^\vee)\rar \bc.
\]

To show $\SC^1(\pi_f)\neq 0$, one needs a cycle class $[Z]\in \SC^1(\pi_f)$ and a closed form $\Omega\in \HH^4_c(\pi_f^\vee)$ satisfying $\int_Z \Omega\neq 0$. However, compactly supported closed forms are hard to construct. As a substitute, rapidly decreasing differential forms are easier to construct and define cohomology groups that are isomorphic to the de Rham cohomology with compact support (see \cite{borel1980}). We propose the following Proposition (See Prop. ~\ref{formOmega} in Sect. ~\ref{pairng}) and show that it is sufficient for $\SC^1(\pi_f)\neq 0$.

\begin{proposition*}
When $\mathrm{H}^{1,1}(\pi_f)$ is nonzero, there exists a special divisor $Z$ and a $\bk_f$-finite rapidly decreasing closed form $\Omega$ on $M$ such that \emph{(i)} $<\mathrm{H}^{1,1}(\pi_f^\prime),\Omega>=0$ for all $\pi_f^\prime\neq \pi_f$, \emph{(ii)} $\int_Z \Omega\neq 0$.
\end{proposition*}

\item[Step 3.] To verify the above proposition, we utilize the automorphic description of $\HH^{1,1}(M,\bc)$ in \cite{Weiss92}. When $\mathrm{H}^{1,1}(\pi_f)\neq 0$, there exists a unique irreducible unitary $\GSp_4(\br)$-representation $\pi_\infty$ such that $\pi=\pi_\infty\times \pi_f$ occurs in the discrete spectrum.  All such $\pi$ are the character twist of three types of basic representations (See Thm. ~\ref{wei_coho}). It suffices to treat the basic representations and, briefly speaking, they are
\begin{itemize}
\item[(I)] a Siegel-type CAP representation of $\PGSp_4(\ba)\cong \SO(3,2)(\ba)$,
\item[(II)] a residue representation of Siegel-type Eisenstein series,
\item[(III)] the trivial representation $1$.
\end{itemize}
We prove the Proposition  for these three types in Section ~\ref{np1}, ~\ref{np2} and ~\ref{np3}.

When $\pi$ is of type I, $\pi_\infty$ is a cohomological parameter $\pi^{2+}$ (see Sect. ~\ref{coho-rep}) and $\pi$ is the global theta lift of an irreducible cuspidal metaplectic $\SL_2(\ba)$-representation. These two facts select out certain non-split $\SO_4\subset \SO(3,2)$ so that the automorphic periods
\[
\mpp(\varphi):=\int_{\SO_4(\bq)\backslash \SO_4(\ba)}\varphi(h)dh,\quad \varphi\in \pi
\]
can be nonzero. We let $Z$ be the Hilbert modular surface associated to this $\SO_4$ and consider forms $\Omega\in\HH^{2,2}(\fg,K_\infty,\pi)$. Condition (i) is satisfied automatically. The cohomological periods $\int_Z \Omega$ are related to the values of $\mpp$ on a subspace $v_0\times \pi_f\subset \pi$, where $v_0$ is a special vector in $\pi_\infty$. We use the Vogan-Zuckerman description of $\pi^{2+}$ to deduce that when $\mpp$ is nonzero on $\pi$, it must be nonzero on $v_0\otimes \pi_f$, whence certain cohomological period is nonzero. This subtle phenomenon happens partly because the $K_\infty$-type containing $v_0$ is minimal in $\pi^{2+}$.


When $\pi$ is of type II or III, the candidate for the cycle $Z$ is a product of modular curves and associated to a subgroup $\SO(2,2)\subset \SO(3,2)$. Let $\mathrm{P}$ be the Siegel parabolic subgroup in case II and the Borel subgroup in case III. We specify certain type of closed degree-$4$ form $\eta$ on the space $\mathrm{P}(\bq)\backslash \GSp_4(\ba)/K_\infty$. Let $L$ denote the left translation on $\GSp_4(\br)$. We follow Harder's method of Eisenstein cohomology and construct
\[
 E(\eta)=\sum_{\gamma\in \mathrm{P}(\bq)\backslash \mathrm{GSp}_4(\bq)} L_\gamma^*\eta.
\]
When $\eta$ is carefully chosen, the form $E(\eta)$ would be closed and meet Condition (i) and (ii).
\end{itemize}

The results of this paper were obtained in 2007 and its writing was completed in 2008 when the author visited University of Wisconsin at Madison. 

We later heard that He and Hoffman \cite{hh2012} independently proved a result similar to Theorem ~\ref{picardgroup} in the classical setting using a different method. As a comparison, \cite{hh2012} proves the result by applying the Kudla-Millson theory \cite{kudla_millson86, kudla_millson87, kudla_millson88}; we do not rely on the work of Kudla-Millson but take the direct approach of constructing rapidly decreasing cohomological forms and pairing them with concrete cycles. Such an explicit study of period pairing is sufficient to show that Tate classes on Siegel 3-folds are from special cycles; additionally, the explicit construction of representatives of the Eisenstein cohomology could be useful for other purposes.

\section{notations}\label{sect_notation}

Let $\ba$ be the ring of adels over $\bq$ and $\ba_f$ be the subring of finite adels. Let $\ba^\times$ be the multiplicative group of invertible elements in $\ba$ and $\ba_1^\times$ denote the subgroup consisting of norm $1$ elements. For $a\in \bq^\times$ (resp. $\bq_p^\times$), let $\chi_a=<a,\cdot>$ be the associated quadratic character of $\bq^\times\backslash \ba^\times$ (resp. $\bq_p^\times$), where $<,>$ is the Hilbert symbol. For a character $\psi$ of $\bq\backslash \ba$ (resp. $\bq_p^\times$), $\psi_a$ denotes the $a$-twist $\psi_a(x)=\psi(ax)$. 

For a reductive group $H$ over $\ba$, let $H(\br)^+$ denote the identity component of $H(\br)$ with respect to the real topology. Set $\br_+:=\{t\in \br: t>0\}$ and embed it into $\ba^\times$ by sending $t$ to $(t,1_{\ba_f})$. Put $c=\mathrm{Vol}(\bq^\times \br_+\backslash \ba^\times)=\mathrm{Vol}(\bq^\times\backslash \ba^\times_1)$.

\subsection{The group $\mathrm{GSp}_4$}\label{group}

Set $J_n=\smalltwomatrix{}{I_n}{-I_n}{}$ and define
\begin{align*}
\mathrm{GSp}_{4}&=\{g\in \mathrm{GL}_{4}: \tran{g}J_2 g=\nu(g) J_2,\, \nu(g)\in \mathrm{GL}_1\}.
\end{align*}
One calls $\nu(g)$ the similitude character. Let $Z$ denote the center of $\GSp_4$. Set $\Sp_4=\Ker \nu$ and $\PGSp_4=\GSp_4/Z$.

There are three types of parabolic subgroups in $\GSp_4$: the Borel subgroup $B$, the Siegel parabolic subgroup $P$, and the Klingen parabolic subgroup $Q$. We fix a choice of them as below,
\[
B=\left\{\left(\begin{smallmatrix}
* &* &* &*\\
 &* &* &*\\
 & &* & \\
 & &* &*\\
\end{smallmatrix}\right)\right\},\quad 
P=\left\{\left(\begin{smallmatrix}
* &* &* &*\\
* &* &* &*\\
 & &* &*\\
 & &* &*\\
\end{smallmatrix}\right)\right\},\quad 
Q=\left\{\left(\begin{smallmatrix}
* &* &* &*\\
* &* &* &*\\
 & &* &\\
 &* &* &*\\
\end{smallmatrix}\right)\right\}.
\]

We also fix a maximal compact subgroup $\bk=\prod_v \bk_p$ of $\GSp_4(\ba)$, with $\bk_\infty=\bk_\br\cdot\{\smalltwomatrix{I_2}{}{}{\pm I_2}\}$ and $\bk_p=\GSp_4(\bz_p)$ when $p$ is finite. Here $\bk_\br$ is specified by its Lie algebra as in Section ~\ref{algebra}. Write $\bk_f=\prod_{p<\infty}\bk_p$.

There is an exceptional isomorphism $\GSp_4\cong \mathrm{GSpin}(V)$, where $V$ is a $5$-dimensional quadratic space over $\bq$ of Witt index $2$ and determinant $1$. To make the isomorphism explicit, set $w=\smalltwomatrix{}{1}{-1}{}$ and 
\begin{align*}
&V=\{Y=\smalltwomatrix{X}{x^\prime w}{-x^{\prime\prime}w}{\leftup{t}{X}}|x^\prime, x^{\prime\prime}\in \bq, X\in M_{2\times 2}(\bq), \mathrm{Tr}(X)=0\},\\
&q(Y)=\frac{1}{4}\mathrm{Tr}(Y^2).
\end{align*}
$\mathrm{GSp}_4$ acts on $V$ by $g\circ Y=gYg^{-1}$ and this action induces an isomorphism $\mathrm{PGSp}_4\isoto \mathrm{SO}(V)$. which is then lifted to an isomorphism $\mathrm{GSp}_4\isoto \mathrm{GSpin}(V)$.

\subsection{The Lie algebra of $\mathrm{GSp}_4(\br)$}\label{algebra}

Put $\fg=\Lie(G(\br))\otimes \bc$, $\fg_0=\Lie(\mathrm{Sp}_4(\br))\otimes \bc$, and $\fz=\Lie(Z(\br))\otimes \bc$. There is $\fg=\fg_0\oplus \fz$. We describe the structure of $\fg_0$ as a complex semi-simple Lie algebra:
\begin{itemize}
\item[(i)] The Cartan involution $\theta: X\rar -\tran{X}$ leads to a Cartan decomposition $\fg_0=\fk+\fp$, with $\theta$ acting on $\fk$ and $\fp$ by $1$ and $-1$ respectively.
\begin{align*}
&\fk=\{X\mid \theta(X)=X\}=
\left\{\smalltwomatrix{s_0J_1}{S}{-S}{s_0J_1}\mid s_0\in\br,\ S\in \mathrm{Sym}_{2\times2}(\bc)\right\},\\
&\fp=\{X\mid \theta(X)=-X\}=
\left\{\smalltwomatrix{S_1}{S_2}{S_2}{-S_1}\mid S_1,S_2\in
\mathrm{Sym}_{2\times2}(\bc) \right\}.
\end{align*}
$\ft=\bc H\oplus \bc J$ is a Cartan subalgebra of $\fg_0$ contained in $\fk$, where
\begin{align*}
&H=\left(\begin{smallmatrix}
 &1 & &\\
 -1& & &\\
 & & &1\\
 & &-1 &
 \end{smallmatrix}\right),\ \
J=\left(\begin{smallmatrix}
 & &1 &\\
 & & &1\\
 -1& & &\\
 &-1 & &
 \end{smallmatrix}\right).
\end{align*}

\item[(ii)] Let $\alpha$ and $\beta$ be linear functionals on $\ft$ given by
\begin{align*}
&\alpha(n_1H+n_2J)=-2n_1i,\\
&\beta(n_1H+n_2J)=-2n_2i.
\end{align*}
Both $\fp$ and $\fk$ are the direct sum of root spaces (with respect to $\ft$),
\begin{align*}
&\fp=V_{-\alpha+\beta}\oplus V_{\alpha-\beta}\oplus
 V_{\beta}\oplus V_{-\beta}\oplus
 V_{\alpha+\beta}\oplus V_{-\alpha-\beta},\\
 &\fk=\ft\oplus \bc V_\alpha \oplus \bc V_{-\alpha},
\end{align*}
where the subscripts denote the roots.

\item[(iii)] Put $\fk_\br=\fk\cap \Lie(\mathrm{Sp}_4(\br))$ and $\ft_\br=\ft\cap \Lie(\mathrm{Sp}_4(\br))=\br H\oplus \br J$, then $\ft_\br$ is a Cartan subalgebra of $\fk_\br$. There is
\[
\fk_\br= \br J\oplus (\bc H\oplus V_\alpha\oplus V_{-\alpha})\cap \fk_\br.
\]
$\br J$ is the center of $\fk_\br$ and $(\bc H\oplus V_\alpha\oplus V_{-\alpha})\cap \fk_\br\cong \fs\fu_2$.
\end{itemize}

Let $K_\br$ be the analytic subgroup of $\mathrm{Sp}_4(\br)$ with Lie algebra $\fk_\br$ and $T$ be the analytic subgroup with Lie algebra $\ft_\br$. 
The finite-dimensional irreducible complex representations of $K_\br$ are exactly the finite-dimensional irreducible complex representations of $\fk$. They are parameterized by highest weights and we denote by $\delta_\gamma$ the one with highest weight $\gamma$. The following lemma is easy to verify.
\begin{lemma}\label{hw}
A linear functional $\gamma:\ft\rar \bc$ is the highest weight of a finite-dimensional irreducible complex representation of $\fk_\br$ if and only if it is of the form $\gamma=n_1\alpha+n_2\beta, n_1\in \frac{1}{2}\bz_{\geq 0}, n_2\in \frac{1}{2}\bz$.
\end{lemma}


\subsection{Siegel $3$-folds}\label{3fold}

We follow Deligne \cite{Deligne71} to define the Siegel $3$-folds associated to the $\bq$-group $G:=\GSp_4$. Fix a $\br$-group homomorphism
\begin{align*}
 h:\Res_\br^\bc \bc^\times &\rar G(\br)\\
 x+iy &\rar \smalltwomatrix{xI_2}{yI_2}{-yI_2}{xI_2}.
\end{align*}
The centralizer of $h$ in $G(\br)$ is $K_\infty=Z(\br)K_\br$. For a compact open subgroup $K_f$ of $G(\ba_f)$, the Siegel $3$-fold $\mkf$ of level $K_f$ is a $3$-dimensional quasi-projective variety over $\bq$ whose set of complex points are
\[
 \mkf(\bc)=G(\bq)\backslash G(\ba)/K_\infty K_f.
\]
The family $\{\mkf\}$ forms an inverse system of $\bq$-varieties and its inverse limit $M:={\ilim}_{K_f} \mkf$ is a $\bq$-scheme (not of finite type).

\subsection{Special cycles}\label{cycle}

We follow \cite{kudla1997} to define special cycles on Siegel $3$-folds. Recall the isomorphism $G\cong \mathrm{GSpin}(V)$. For a positive-definite subspace $V_0$ of $V$, regard $G_{V_0}:=\mathrm{\mathrm{\mathrm{GSp}in}}(V_0^\perp)$ as a $\bq$-subgroup of $G$.
\begin{itemize}
\item[(i)] Choose an element $g_\infty\in G(\br)$ such that
\[
L_\infty:=g_\infty K_\infty g_\infty^{-1} \cap G_{V_0}(\br)
\]
contains a maximal connected compact subgroup of $G_{V_0}(\br)$. All such choices then form a double coset $G_{V_0}(\br) g_\infty K_\infty$.

\item[(ii)] For a compact open subgroup $K_f\subset G(\ba_f)$ and an element $g_f\in G_{V_0}(\ba_f)$, put $L_f=g_fK_fg_f^{-1}\cap G_{V_0}(\ba_f)$.

\item[(iii)] Let $\mz_{V_0,g_f,K_f}$ be the Shimura variety associated to $G_{V_0}$ of level $L_f$. It is quasi-projective and defined over $\bq$, with $\mz_{U,g_f,K_f}(\bc)=G_{V_0}(\bq)\backslash G_{V_0}(\ba)/L_\infty L_f$ and $\dim \mz_{V_0,g_f,K_f}=\dim V_0$.
\end{itemize}
There is a natural morphism $i_{g_f,K_f}:\mz_{V_0,g_f,K_f}\rar \mkf$; over complex points, it is given by $\mz_{V_0,g_f,K_f}(\bc) \rar \mkf(\bc)$, $(x_\infty, x_f) \rar (x_\infty g_\infty, x_f g_f)$. We write $Z_{V_0,g_f,K_f}$ for the $\bq$-cycle $i_{g_f,K_f}(\mz_{V_0,g_f,K_f})$.

\begin{definition}
$i=1,2,3$. $\SC^i(\mkf)$ is the subgroup of the Chow group $\CH^i(\mkf\otimes_\bq \bbq)$ generated by connected components of $Z_{V_0,g_f,K_f}\otimes_\bq \bar{\bq}$, with $g_f$ running over $G(\ba_f)$ and $V_0$ running over positive-definite subspaces of dimension $i$.
\end{definition}

The family $\{\SC^i(\mkf)\}$ forms a direct system: for two compact open subgroups $K_f \subset K_f^\prime$, the projection $M_{K_f} \rar M_{K_f^\prime}$ is flat and induces a pull-back homomorphism $\SC^i(M_{K_f^\prime})\rar \SC^i(\mkf)$. The direct limit
\[
\SC^i(M):=\dlim\ _{K_f} \SC^i(\mkf)
\]
is a $G(\ba_f)$-module. For $g_f\in \mathrm{\mathrm{GSp}in}(V)_{\ba_f}$, the translation map $\rho(g_f)_{K_f}: \mkf \rar M_{g_f^{-1}K_fg_f}$, $(x_\infty, x_f)\rar (x_\infty, x_fg_f)$ induces an isomorphism $\rho(g_f)_{K_f}^\ast: \SC^i(M_{g_f^{-1}K_fg_f})\rar \SC^i(\mkf)$, then $g_f$ acts on $\SC^i(M)$ by $\rho(g_f)^\ast=\dlim_{K_f}\rho(g_f)_{K_f}^\ast$. When $K_f$ is neat, $\SC^i(M)^{K_f}=\SC^i(\mkf)$.

We comment that the group $G_{V_0}$ has a realization other than $\GSpin$.

(i) $V_0=\bq v$, then $G_{V_0}\cong \GL_2^\prime(F)$, with $F=\bq(\sqrt{q(v)})$ and
\[
\mathrm{GL}_2^\prime(F):=\{g\in \mathrm{GL}_2(F)|\det g\in \bq^\times\};
\]
Specifically, when $q(v)\in {\bq^\times}^2$, there are $F=\bq\oplus \bq$ and $\GL_2^\prime(F)=\{(g_1,g_2)\in \GL_2\times \GL_2: \det g_1=\det g_2\}$; the connected components of $Z_{\bq v,g_f,K_f}$ are essentially the product of two modular curves associated to $\GL_2$.

(ii) $\dim V_0=2$. There exists a indefinite quaternion algebra $D$ over $\bq$ such that $G_{V_0}\cong D^\times$. Here indefinite means $D(\br)\cong M_{2\times 2}(\br)$.
\begin{remark}
Let $\SC^1_{s}(M)$ (resp. $\SC^1_{ns}(M)$) denote the subspace of $\SC^1(M)$ spanned by connected components of $Z_{\bq v,g_f, K_f}$ with $q(v)\in {\bq^\times}^2$ (resp. $q(v)\in \bq_+\backslash {\bq^\times}^2$). We call cycles in $\SC^1_s(M)$ split divisors and cycles in $\SC^1_{ns}(M)$ non-split divisors.
\end{remark}

\subsection{CAP representations}\label{parabolic} 


Let $\tau$ be an irreducible cuspidal automorphic representation of $\mathrm{GL}_2(\ba)$ and $V_\tau\subset \ma_{cusp}(\PGL_2)$ be the underlying space of $\tau$. Let $\chi$ be a character of $\ba^\times/\bq^\times$ and $z\in \bc$. Let $\Pi(\tau\boxtimes \chi,z)$ denote the induced representation of $G(\ba)$ consisting of smooth $\bk$-finite functions $f:G(\ba)\rar V_\tau$ that satisfy
\[
f\left(\smalltwomatrix{A}{u}{}{x\leftup{t}{A^{-1}}}g\right)=|\frac{\det A}{x}|^{\frac{3}{2}+z}\chi(x)\tau(A)f(g)
\]
for $u\in \mathrm{Sym}_{2\times 2}(\ba)$, $A\in \GL_2(\ba)$, and $x\in \ba^\times$.

\begin{definition}
An irreducible cuspidal automorphic representation $\pi$ of $\mathrm{GSp}_4(\ba)$ is called CAP of Siegel type $(\tau\boxtimes \chi,z)$ if there exists an irreducible constituent $\Pi$ of $\Pi(\tau\boxtimes \chi,z)$ such that $\pi_p\cong \Pi_p$ for almost all $p$.
\end{definition}

To describe CAP representations of Siegel type, it is necessary to use the theta correspondence between $\PGSp_4, \PGL_2$ and the metaplectic group $\wtilde{\SL}_2$. We refer to Section ~\ref{theta} for notions of theta correspondence.

Choose a non-trivial character $\psi$ of $\bq\backslash \ba$. Let $\ma_{00}(\wtilde{\SL}_2)$ be the space of cuspidal automorphic forms on $\wtilde{\SL}_2(\ba)$ that are orthogonal to elementary theta series. Waldspurger \cite{Wald80} \cite{Wald91} proved a packet decomposition
\[
\ma_{00}(\wtilde{\SL}_2)=\sqcup_{\tau\subset \ma_{cusp}(\PGL_2)} \Wd_\psi(\tau),
\]
where $\tau$ runs over irreducible cuspidal automorphic represnetation of $\PGL_2(\ba)$ and the Waldspurger packet $\Wd_\psi(\tau)$ is the collection of all non-zero global theta lifts $\Theta_{\wtilde{\SL}_2\times \PGL_2}(\tau\otimes \chi_a,\psi_a)$, $a\in \bq^\times$. Define the $\ell_\psi$-Whittaker functional on $\ma(\wtilde{\SL}_2)$ by
\[
\ell_{\psi}(\varphi):=\int_{\bq\backslash \ba} \varphi\big(\smalltwomatrix{1}{n}{}{1}\big)\psi(-n)dn.
\]
For $\sigma\in \Wd_\psi(\tau)$ and $a\in \bq^\times$, there is
\begin{equation*}
\Theta_{\wtilde{\SL}_2\times \PGL_2}(\sigma,\psi_a)=
\begin{cases}
\tau\otimes \chi_a, &\ell_{\psi_a}|_{\sigma}\neq 0,\\
0, &\ell_{\psi_a}|_{\sigma}=0.
\end{cases}
\end{equation*}

\begin{theorem}\cite{PS83}\label{ps_cap}
$\pi$ is CAP of Siegel type $(\tau\boxtimes \chi, z)$ if and only if
\begin{itemize}
\item[(i)] $z=\pm\frac{1}{2}$ and $\tau$ has a trivial central character,
\item[(ii)] $\pi=\Theta_{\wtilde{\SL}_2\times \PGSp_4}(\sigma,\psi)\otimes \chi$, where $\sigma$ is an irreducible cuspidal automorphic representation of $\wtilde{\SL}_2(\ba)$ belonging to the Waldspurger packet $\Wd_\psi(\tau)$ and satisfying $\Theta_{\wtilde{\SL}_2\times \PGL_2}(\sigma,\psi)=0$.
\end{itemize}
\end{theorem}
CAP representation of Siegel type all occur with multiplicity one in the discrete spectrum $L^2_{\disc}(G(\bq)Z(\br)^+\backslash G(\ba))$.

\subsection{Cohomological parameters}\label{coho-rep}

A cohomological parameter of $G$ is an irreducible unitary $(\fg,K_\infty)$-module that has non-trivial $(\fg,K_\infty)$-cohomology. By the Vogan-Zuckerman theory \cite{VZ84}, one can determine all seven cohomological parameters of $G$. Five of them contribute to $\mathrm{H}^2$ and, among these five, four contribute to $\mathrm{H}^{1,1}$. The four parameters are $\{1, \sgn\circ \nu, \pi^{2+},\pi^{2-}\}$ and their nonzero $(\fg,K_\infty)$-cohomology groups are
\begin{align*}
&H^{1,1}(\fg,K_\infty,\pi^{2,\pm})=H^{2,2}(\fg,K_\infty,\pi^{2,\pm})=\bc,\\
&H^{i,i}(\fg,K_\infty,1)=\bc\ (i=0,1,2,3),\\
&H^{i,i}(\fg,K_\infty,\sgn)=\bc\ (i=0,1,2,3).
\end{align*}

\subsubsection{The Langlands parameter of $\pi^{2\pm}$}\label{sss_lppi2}

Let $\tau_\infty$ be an irreducible tempered unitary representation of $\GL_2(\br)$ on the space $V_{\tau_\infty}$. Let $\chi_\infty$ be a character of $\br^\times$ and $z\in \bc$. Let $I^\infty_{P,\tau_\infty\boxtimes \chi_\infty, z}$ be the smooth induced representation of $G(\br)$ consisting of smooth functions $f:G(\br)\rar V_{\tau_\infty}$ that satisfy
\[
f_\infty\left(\smalltwomatrix{I_2}{u}{}{I_2}\smalltwomatrix{A}{}{}{x\leftup{t}{A^{-1}}}g\right)=\big|\frac{\det A}{x}\big|^{z+\frac{3}{2}}\chi_\infty(x)\tau_\infty(A)f_\infty(g)
\]
for $A\in \GL_2(\br), u\in \mathrm{Sym}_{2\times 2}(\br)$ and $x\in \br^\times$. Let $I_{P,\tau_\infty\boxtimes \chi_\infty, z}$ be the underlying $(\fg,K_\infty)$-module of $I^\infty_{P,\tau_\infty\boxtimes \chi_\infty, z}$ consisting of $K_\infty$-finite functions. When $\re(z)>0$, $I_{P,\tau_\infty\boxtimes \chi_\infty, z}$ is irreducible for almost all $z$ and, when it is reducible, has a unique quotient. We set $J_{P,\sigma\boxtimes \chi, z}$ to be $I_{P,\sigma\boxtimes \chi, z}$ in the former case and to be the unique quotient in the latter case.

The calculation in \cite{VZ84} allows one to identify $\pi^{2\pm}$ as a Langlands quotient of the above type. Let $\fD_{2n}$ ($n\geq 1$) be the discrete series representation of $\GL_2(\br)$ with trivial central character and of weight $2n$. In other words, $\fD_{2n}$ is the unique subrepresentation of the $\GL_2(\br)$-representation induced from the quasi-character $\smalltwomatrix{a_1}{*}{}{a_2}\rar |\frac{a_1}{a_2}|^{n-\frac{1}{2}}$. There is $\{\pi^{2+},\pi^{2-}\}=\{J_{P,\fD_4\otimes 1, \frac{1}{2}},
J_{P,\fD_4\otimes \sgn, \frac{1}{2}}\}$ and we set
\[
\pi^{2+}=J_{P,\fD_4\otimes 1, \frac{1}{2}},\quad \pi^{2-}=J_{P,\fD_4\otimes \sgn, \frac{1}{2}}.
\]

\subsubsection{The structure of $H^2(\fg,K_\infty,\pi^{2+})$}\label{coho_liealg}

We keep the notations as in Section ~\ref{algebra}. For a vector space $\fa$ over $\bc$, let $\fa^\ast$ denote its dual $\mathrm{Hom}(\fa,\bc)$. For $j\geq 0$, there is a canonical isomorphism $(\wedge^j \fa)^\ast \cong \wedge^j \fa^\ast$ and we identify $(\wedge^j \fa)^\ast$ with $\wedge^j \fa^\ast$.

Now we compute the $(\fg,K_\infty)$-cohomology of $\pi^{2+}$. Set $B_0=B\cap \mathrm{Sp}_4$, $\fb_0=\Lie(B_0(\br))\otimes \bc$ and $\tilde{\fk}:=\mathrm{Lie}(K_\infty)=\fk\oplus \fz$. There is $\fg=\fb_0\oplus \wtilde{\fk}$, whence $\fb_0\cong \fg/\tilde{\fk}\cong \fp$ and $\fb_0^\ast \cong (\fg/\tilde{\fk})^\ast \cong \fp^\ast$. $K_\infty$ and $\tilde{\fk}$ act on $\fb_0$ and $\fb_0^*$ in an according way. There is
\begin{align*}
H^2(\fg,K_\infty,\pi^{2+})
=&\Hom_{K_\infty}(\wedge^2(\fg/\wtilde{\fk}),\pi^{2+})\\
=&\Hom_{\fk}(\wedge^2\fb_0,\pi^{2+})\\
=&\big((\wedge^2\fb_0)^*\otimes \pi^{2+}\big)^{\fk}\\
=&\big(\wedge^2\fb_0^*\otimes \pi^{2+}\big)^{\fk}
\end{align*}

The $\fk$-module structure of $\wedge^2\fb_0^*$ can be determined as below.
\begin{itemize}
\item[(i)] Chose the following explicit basis of $\fb_0$:
\begin{align*}
&a=\left(\begin{smallmatrix}
 1& & &\\
 &1 & &\\
 & &-1 &\\
 & & &-1
 \end{smallmatrix}\right),\ \
h=\left(\begin{smallmatrix}
 1& & &\\
 &-1 & &\\
 & &-1 &\\
 & & &1
 \end{smallmatrix}\right),\ \
n_0=\left(\begin{smallmatrix}
 0&1 & &\\
 &0 & &\\
 & &0 &\\
 & &-1 &0
 \end{smallmatrix}\right),\ \ \\
&n_1=\left(\begin{smallmatrix}
 0& &1 &\\
 &0 & &1\\
 & &0 &\\
 & & &0
 \end{smallmatrix}\right),\ \
n_2=\left(\begin{smallmatrix}
 0& &1 &\\
 &0 & &-1\\
 & &0 &\\
 & & &0
 \end{smallmatrix}\right),\ \
n_3=\left(\begin{smallmatrix}
 0& & &1\\
 &0 &1 &\\
 & &0 &\\
 & & &0
 \end{smallmatrix}\right).
\end{align*}
The weight space decomposition of $\fb_0$ with respect to the action of $\ft$ is $\fb_0=\bc e_{-\alpha-\beta}\oplus \bc e_{-\beta}\oplus \bc e_{\alpha-\beta}\oplus \bc e_{-\alpha+\beta}\oplus \bc e_{\beta}\oplus \bc e_{\alpha+\beta}$, where the subscripts denote the weights and 
\begin{align*}
&e_{-\alpha-\beta}=\frac{1}{2}h+n_0i+n_2i-n_3,
&e_{\alpha+\beta}=\frac{1}{2}h-n_0i-n_2i-n_3,\\
&e_{\alpha-\beta}=\frac{1}{2}h-n_0i+n_2i+n_3,
&e_{-\alpha+\beta}=\frac{1}{2}h+n_0i-n_2i-n_3,\\
&e_{\beta}=\frac{1}{2}a-n_1i, & e_{-\beta}=\frac{1}{2}a+n_1i.
\end{align*}

\item[(ii)] 
$\wedge^2\fb_0, \wedge^4\fb_0, \wedge^2\fb_0^*, \wedge^4\fb_0^*$ are isomorphic as $\fk$-modules and each is the direct sum of five irreducible $\fk$-submodules with highest weights $0,\alpha-2\beta, \alpha, \alpha+2\beta, 2\alpha$ respectively.

\item[(iii)] Let $\{e^*_{-\alpha-\beta}, e^*_{-\beta}, e^*_{\alpha-\beta}, e^*_{-\alpha+\beta}, e^*_{\beta}, e^*_{\alpha+\beta}\}$ be the basis of $\fb_0^*$ dual to the basis $\{$ $e_{-\alpha-\beta}$, $e_{-\beta}$, $e_{\alpha-\beta}$, $e_{-\alpha+\beta}$, $e_{\beta}$, $e_{\alpha+\beta}$$\}$ of $\fb_0$. The irreducible $\fk$-submodule of $\wedge^2 \fb^*_0$ with highest weight $2\alpha$ is of dimension $5$ and spanned by
\begin{align*}
&\eta_2=e^*_{-\alpha-\beta}\wedge e^*_{-\alpha+\beta},\\
&\eta_1=e^*_{-\alpha-\beta}\wedge e^*_{\beta}+e^*_{-\beta}\wedge
e^*_{-\alpha+\beta},\\
&\eta_0=e^*_{-\alpha-\beta}\wedge
e^*_{\alpha+\beta}+2e^*_{-\beta}\wedge
e^*_{\beta}+e^*_{\alpha-\beta}\wedge e^*_{-\alpha+\beta},\\
&\eta_{-1}=e^*_{\alpha-\beta}\wedge e^*_{\beta}+e^*_{-\beta}\wedge
e^*_{\alpha+\beta},\\
&\eta_{-2}=e^*_{\alpha-\beta}\wedge e^*_{\alpha+\beta}.
\end{align*}
The vector $\eta_j (-2\leq j\leq 2)$ is of weight $j\alpha$.
\end{itemize}

\begin{lemma}\label{cohopi2}
\begin{itemize}
\item[(i)]The irreducible $\fk$-submodules of $\pi^{2+}$ are of highest weights $m_1\alpha+m_2\beta$, with $m_1\in \bz_{\geq 2},m_2\in\bz$ and $m_1\equiv m_2\mod 2$. Specifically, the irreducible $\fk$-submodule of $\pi^{2+}$ with highest weight $2\alpha$ occurs with multiplicity one in $\pi^{2+}$.
\item[(ii)]Let $\pi^{2+}(\delta_{2\alpha})$ be the irreducible $\fk$-submodule of $\pi^{2+}$ with highest weight $2\alpha$. There is a basis $\{v_j|-2\leq j\leq 2\}$ of $\pi^{2+}(\delta_{2\alpha})$ such that $v_j$ is of weight $j\alpha$ and $\HH^2(\fg,K_\infty,\pi^{2+})=\bc\sum_{j=-2}^2 v_{-j}\otimes\eta_j$.
\end{itemize}
\end{lemma}
\begin{proof}
Part (i) is a direct application of the Vogan-Zuckerman theory to $G$. For (ii), recall that $H^2(\fg,K_\infty,\pi^{2+})=\Hom_{\fk}(\wedge^2\fb_0,\pi^{2+})$; by (i) and the structure of $\wedge^2 \fb_0$, this space is $1$-dimensional and an $\fk$-invariant homomorphism from $\wedge^2\fb_0$ to $\pi^{2+}$ can only
happen between their $\fk$-submodules with highest weight $2\alpha$. Choose a basis $\{v_j|-2\leq j\leq 2\}$ of $\delta$ such that $v_j$ is of weight $j\alpha$, then such a homomorphism is of the form $\sum_{i=-2}^2 c_j v_{-j}\otimes \eta_j$, where $c_j$ are nonzero numbers. Replacing $v_j$ by $c_{-j}v_j$, one gets the desired basis of $\delta$.
\end{proof}

\subsection{Theta correspondence}\label{theta}

We prepare here the relevant notions of theta lifting for the discussion of CAP representations in Section ~\ref{parabolic}.

\subsubsection{Local theta correspondence}

Let $k$ be a local field of characteristic zero and $\psi$ a non-trivial character of $k$. For a reductive group $\mathrm{G}$ over $k$, let $\mathrm{Irr}(\mathrm{G})$ denote the set of irreducible admissible representations of $\mathrm{G}$ when $k$ is non-archimedean and the set of irreducible admissible $(\mathrm{Lie}(\mathrm{G})\otimes_\br \bc,\mathrm{K})$-modules when $k$ is archimedean. ($\mathrm{K}$ denotes a maximal compact subgroup of $\mathrm{G}$ in case of $k$ being archimedean.)

Let $\wtilde{\SL}_2(k)$ be the $2$-fold metaplectic cover of $\SL_2(k)$; it is $\SL_2(k)\times \bz_2$ as a set and the upper triangular unipotent group in $\SL_2(k)$ has a lift to $\wtilde{\SL}_2(k)$. 
Let $(U,q)$ be a quadratic space over $k$ of odd dimension and $\ms(U)$ be the space of Bruhat-Schwartz functions on $U$. The Weil representation $\omega:=\omega_{\psi,U}$ of $\wtilde{\SL}_2(k)\times \OO(U)$ on $\ms(U)$ is given by:
\begin{align*}
&\omega(h)\phi(X)=\phi(h^{-1}X),\quad h\in O(U),\\
&\omega\big(\smalltwomatrix{1}{n}{}{1},\epsilon\big)\phi(X)=\epsilon \psi\big(nq(X)\big)\phi(X),\quad n\in k,\, \epsilon\in \bz_2,\\
&\omega\big(\smalltwomatrix{a}{}{}{a^{-1}},\epsilon\big)\phi(X)=\epsilon\chi_{\psi,U}(a)|a|^{\frac{\dim U}{2}}\phi(aX),\quad a\in k^\times,\\
&\omega(\smalltwomatrix{}{1}{-1}{},\epsilon)\phi(X)=\epsilon \gamma(\psi, U)\int_{U}\phi(Y)\phi(q(X,Y))dY.
\end{align*}
Here $q(X,Y)=q(X+Y)-q(X)-q(Y)$ is the associated quadratic pairing, $\gamma(\psi,U)$ is a constant of norm $1$, and $\chi_{\psi,U}:k^\times \rar \mathrm{S}^1$ is a function satisfying $\chi_{\psi,U}(a_1a_2)=\chi_{\psi,U}(a_1)\chi_{\psi,U}(a_2)<a_1,a_2>$. 

Let $\pi\in \mathrm{Irr}(\OO(U))$ and $\sigma\in \mathrm{Irr}(\wtilde{\SL}_2)$ satisfying $\sigma((I_2,\epsilon))=\epsilon$. The maximal $\pi$-isotypic quotient of $\omega$ is of the form $\pi\boxtimes \theta_0(\pi)$ and the maximal $\sigma$-isotypic quotient of $\omega$ is of the form $\theta_0(\sigma)\boxtimes \sigma$, where $\theta_0(\pi)$ and $\theta_0(\sigma)$ are finite-length admissible representations of $\wtilde{\SL}_2(k)$ and $\OO(U)$ respectively. Let $\theta(\pi)$ and $\theta(\sigma)$ be the maximal semi-simple quotient of $\theta_0(\pi)$ and $\theta_0(\sigma)$ respectively. The Howe Duality Conjecture asserts that $\theta(\pi)$ and $\theta(\sigma)$ are irreducible. The Howe Duality Conjecture has been proved for general reductive dual pairs when $k$ is not a dyadic field; when $k$ is dyadic, it is also known for the current pair $(\SL_2, \OO(U))$ when $\dim U=1, 3$ and $5$.

When $\theta(\pi)=\sigma$ and $\theta(\sigma)=\pi$, we say that $\pi$ and $\sigma$ are in local theta correspondence with respect to $\psi$ and call $\sigma$ (resp. $\pi$) the local theta lift of $\pi$ (resp. $\sigma$). We use the notations $\theta(\pi,\psi)$ and $\theta(\sigma,\psi)$ when the role of $\psi$ needs to be emphasized.

Because $\dim U$ is assumed to be odd, each $\pi\in \mathrm{Irr}(\SO(U))$ has two extensions to $\OO(U)$. When $\dim U\geq 3$, at most one of the extensions occurs in the local theta correspondence with $\wtilde{\SL}_2(k)$ and one actually considers local theta correspondence between $\mathrm{Irr}(\SO(U))$ and $\mathrm{Irr}(\wtilde{\SL}_2)$.

\subsubsection{Global theta lifting}\label{sss_gtheta}

Let $F$ be a number field, $\ba_F$ be the ring of adels over $F$, and $\psi$ be a non-trivial character of $\ba_F/F$. For a semisimple group $\mathrm{G}$ over $F$, $[\mathrm{G}]$ denotes the quotient space $\mathrm{G}(F)\backslash \mathrm{G}(\ba_F)$.

Let $\wtilde{\SL}_2(\ba_F)$ be the two-fold metaplectic cover of $\SL_2(\ba_F)$. Let $(U,q)$ be a quadratic space over $F$ of odd dimension and $\ms(U(\ba_F))$ be the space of Bruhat-Schwartz functions on $U(\ba_F)$. Let $\omega:=\omega_{\psi, U}=\otimes_v \omega_{\psi_v, U(F_v)}$ be the global Weil representation of $\wtilde{\SL}_2(\ba_F)\times \OO(U)_{\ba_F}$ on $\ms(U(\ba_F))$ with respect to $\psi$. For $\phi\in \ms(U(\ba_F))$, the associated theta kernel function
\[
\theta_\phi(h,g)=\sum_{\xi\in U(F)} \omega(h,g)\phi(\xi)
\]
is a slowly increasing function on $(\OO(V)_F\backslash \OO(V)_\ba)\times (\Sp_n(F)\backslash \wtilde{\Sp}_n(\ba))$. Specifically, when $\dim U=1$, the functions $\theta_\phi(g)$ are called elementary theta series on $\wtilde{SL}_2(\ba)$.

Suppose that $\dim U\geq 3$. Let $\sigma$ be an irreducible cuspidal representation of $\wtilde{\SL}_2(\ba_F)$ and $\pi$ be an irreducible cuspidal representation of $\SO(U)_{\ba_F}$. We define the global theta lift of $\sigma$ to $\SO(U)_{\ba_F}$ and the global theta lift of $\pi$ to $\wtilde{\SL}_2(\ba_F)$ as
\begin{align*}
&\Theta(\sigma,\psi):=\{\Theta(\phi,\varphi):\phi\in \ms(U(\ba_F)), \varphi\in \sigma\},\\
&\Theta(\pi,\psi):=\{\Theta(\phi,f):\phi\in \ms(U(\ba_F)), f\in \pi\},
\end{align*}
where
\begin{align*}
&\Theta(\phi,\varphi):=\int_{[SL_2]}\overline{\varphi(g)}\Theta_\phi(h,g),\quad \Theta(\phi,f):=\int_{[\SO(U)]}\overline{f(h)}\Theta_\phi(h,g).
\end{align*}

\subsubsection{Theta correspondence between $\wtilde{\SL}_2$, $\PGL_2$ and $\PGSp_4$}\label{sss_pi2+}

There is an isomorphism $\PGL_2\cong \SO(V^\prime)$, where 
\[
V^\prime=\{X\in M_{2\times 2}(\bq):\Tr(X)=0\},\quad q^\prime(X)=-\det X,
\]
Recall that $\PGSp_4\cong \SO(V)$. We refer the theta correspondence between $\wtilde{\SL}_2$ and $\PGL_2$ (resp. $\PGSp_4$) to the theta correspondence between $\wtilde{\SL}_2$ and $\SO(V^\prime)$ (resp. $\SO(V)$).

Choose a non-trivial character $\psi_\infty$ of $\br$. The following lemma gives a description of $\pi^{2+}$ in terms of local theta correspondence. 
\begin{lemma}\label{localtheta}
Set $\sigma_\infty=\theta_{\wtilde{\SL}_2\times \PGL_2}(\fD_{2n},\psi_\infty)$, then $\theta_{\wtilde{\SL}_2\times \PGSp_4}(\sigma_\infty,$ $\psi_\infty)=J_{P,\fD_{2n}\otimes 1,\frac{1}{2}}$ and $\theta_{\wtilde{\SL}_2\times \PGSp_4}(\sigma_\infty,\psi^{-1}_\infty)$ is a discrete series.
\end{lemma}
\begin{proof}
This lemma is a special case of Proposition 5.5 in \cite{wgan2008}.
\end{proof}

\section{Cycle maps and Proof of Theorem ~\ref{main}}\label{cyclemap}

We review the cycle map and prove Theorem \ref{main} based on Theorem ~\ref{picardgroup}. 

\subsection{The maps $\cl_{et}$ and $\cl$}

Let $X$ be an algebraic variety over a number field $F$. For $0\leq i\leq 2\dim X$, let $\mz^i(X_\bbq)$ be the free abelian group generated by codimension-$i$ irreducible closed subvarieties of $X_\bbq$ and $\CH^i(X_\bbq)$ be the quotient group of $\mz^i(X_\bbq)$ modulo the relation of rational equivalence. For a cycle $Z\in \mz^i(X_\bbq)$, $[Z]$ refers to its class in $\CH^i(X_\bbq)$.

Let $\HH^\ast(X_\bbq,\bz_\ell)$ and $\HH^\ast_c(X_\bbq,\bz_\ell)$ be the $\ell$-adic cohomology and $\ell$-adic cohomology with compact support. For a $\bz_\ell$-module $R$, set $\HH^\ast_?(X_\bbq,R)=\HH^\ast_?(X_\bbq,\bz_\ell)\otimes R$. The image of $\HH^\ast_c$ in $\HH^\ast$ is denoted by $\HH^\ast_!$ and called the interior cohomology. Let $\mu_n$ be the multiplicative group of the $n$-th roots of unity in $\bbq$ and set $\bz_\ell(1):=\ilim_n \mu_{\ell^n}$; $\bz_\ell(1)$ is isomorphic to $\bz_\ell$ as a $\bz_\ell$-module but $\Gal(\bbq/\bq)$ acts on it by the cyclotomic character. Let $\bz_\ell(i)$ be the $i$-fold tensor of $\bz_\ell$ and set $\bq_\ell(i)=\bz_\ell(i)\otimes \bq_\ell$. There is a $\Gal(\bbq/F)$-equivariant $\ell$-adic cycle map
\[
\cl_{et}: \CH^i(X_\bbq)\rar \HH^{2i}(X_\bbq,\bz_\ell(i)).
\]

For a finite field extension $E\supset F$ contained in $\bbq$, let $\CH^i(X_E)$ be the subgroup of $\CH^i(X_\bbq)$ generated by irreducible closed subvarieties defined over $E$. There is $\CH^i(X_E)=\CH^i(X_\bbq)^{\Gal(\bbq/E)}$. Elements in $\Ta_E^{2i}(X_\bbq):=\HH^{2i}(X_\bbq,\bq_\ell(i))^{\Gal(\bbq/E)}$ are called degree-$2i$ Tate classes over $E$. The union $\Ta^{2i}(X_\bbq)=\cup_E \Ta_E^{2i}(X_\bq)$ for all finite extensions $E\supset F$ is the space of all degree-$2i$ Tate classes. Tate's conjecture asserts that $\Ta_E(X_\bbq)$ is generated by the image of $\CH^i(X_E)$.

When $X$ is smooth, for an abelian group $R$, let $\mathrm{H}^*(X,R)$ (reps. $\mathrm{H}^*_c(X,R)$) denote the singular cohomology (reps. the singular cohomology with compact support) of $X(\bc)$ with respect to its complex topology. When $R$ is $\br$ and $\bc$, $\mathrm{H}^*(X,R)$ (resp. $\mathrm{H}^*_c(X,R)$) is isomorphic to the de Rham cohomology (resp. de Rham cohomology with compact support) defined with the cochain of differential forms (resp. differential forms with compact support). There is a cycle map
\[
\cl: \CH^i(X_\bbq)\rar \HH^{2i}(X,\bz).
\]
defined with regard to the Poincar\'{e} duality between $\HH^\ast_c$ and $\HH^\ast$.
\begin{remark}\label{rm_comparison}
When $X$ is smooth, by \cite[Expos\'{e} XVI]{SGA4.3}, there are canonical isomorphisms between the etale cohomology with constant sheaf $\bz/n\bz$ and the singular cohomology with coefficients in $\bz/n\bz$.
\[
\HH^j(X_\bbq,\bz/n\bz)\cong \HH^j(X,\bz/n\bz),\quad \HH^j_c(X_\bbq,\bz/n\bz)\cong H^j(X,\bz/n\bz).
\]
The comparison map is functorial and respect Poincar\'{e} duality, whence it is compatible with the cycle map $\cl_{et}$ into $\HH^{2i}(X_\bbq,\bz/n\bz(i))$ and the cycle map $\cl$ into $\HH^{2i}(X,\bz/n\bz)$. By passing to the inverse limit, one gets that $\cl_{et}:\CH^i(X_\bbq)\rar \HH^{2i}(X_\bbq,\bz_\ell(i))$ and $\cl:\CH^i(X_\bbq)\rar \HH^{2i}(X,\bz)$ are compatible with the canonical isomorphism
\[
\HH^{2i}(X_\bbq,\bz_\ell(i))\cong \HH^{2i}(X,\bz_\ell)= \HH^{2i}(X,\bz)\otimes_\bz \bz_\ell
\]
\end{remark}

\subsection{The cycle map on $\mkf$ and $M$}

Let $K_f$ be a neat compact open subgroup of $G(\ba_f)$. Let $\tmkf$ be a smooth projective toroidal compactification of $\mkf$ which is defined over $\bq$ and whose boundary is the union of divisors $B_i$ ($1\leq i\leq m$) with normal crossings. Consider the cycle maps on $\mkf$ and $\tmkf$ respectively,
\begin{align*}
&\cl_{\mkf}:\CH^i(\mkf\otimes_\bq \bbq)\rar \HH^{2i}(\mkf,\bz),\\
&\cl_{\tmkf}:\CH^i(\tmkf\otimes_\bq \bbq)\rar \HH^{2i}(\tmkf,\bz).
\end{align*}
For a codimension-$i$ irreducible closed subvariety of $\mkf\otimes_\bq \bbq$, its Zariski closure $\overline{Z}$ in $\tmkf\otimes_\bq \bbq$ is also irreducible. Extend the map $Z\rar \overline{Z}$ to a homomorphism $\mz^i(\mkf)\rar \mz^i(\tmkf)$ and let $j:\HH^\ast(\tmkf,\bz)\rar \HH^\ast(\mkf,\bz)$ denote the restriction map, then
\begin{equation}\label{comp}
\cl_\mkf([Z])=j\circ \cl_{\tmkf}([\overline{Z}]).
\end{equation}

There is a short exact sequence (see \cite{Weiss88}),
\begin{equation}\label{ses}
 0\lrar \oplus_{i=1}^m \bc[B_i] \overset{\cl_{\tmkf}}{\lrar} \HH^2(\tmkf,\bc)\overset{j}{\lrar} \HH^2(\mkf,\bc)\lrar 0,
\end{equation}
Also, it is well-known that $\HH^1(\mkf,\bc)=0$.

\begin{lemma}\label{cyclelemma}
\begin{itemize}
\item[(i)] $\CH^1_0(\mkf\otimes_\bq \bbq)\otimes_\bz \bq=0$
\item[(ii)] Suppose $Z_1, Z_2\in \mz^1(\mkf\otimes_\bq \bbq)$. If $[Z_1]=[Z_2]$ in $\CH^1(\mkf\otimes_\bq \bbq)\otimes_\bq \bq$, then $[\overline{Z}_1]\in [\overline{Z}_2]+(\oplus_{i=1}^m \bq[B_i])$ in $\CH^1(\tmkf\otimes_\bq \bbq)\otimes_\bq \bq$.
\end{itemize}
\end{lemma}
\begin{proof}
Because $\HH^1(\mkf,\bc)=0$, there is $\HH^1(\tmkf,\bc)=0$. So the Picard variety is trivial and therefore $\CH^1_0(\tmkf\otimes_\bq \bbq)=0$.

We first verify (i). Suppose $[Z]\in \CH^1_0(\mkf\otimes_\bq \bbq)$, then $j\circ \cl_{\tmkf}([\overline{Z}])=0$ by (\ref{comp}). By (\ref{ses}), $\cl_{\tmkf}([\overline{Z}])$, considered as an element in $\HH^2(\tmkf,\bq)$, belongs to $\oplus_i \bq\cdot \cl_{\tmkf}([B_i])$. Hence there exists a nonzero integer $n$ and integers $n_i$ such that $n\cdot \cl_{\tmkf}([\overline{Z}])-\sum_i n_i \cdot \cl_{\tmkf}([B_i])$ is equal to zero in $\HH^2(\tmkf,\bz)$. Because $\CH^1_0(\tmkf\otimes_\bq \bbq)$ vanishes, the cycle $n \overline{Z}-\sum_i n_i B_i$ is rationally equivalent to zero and of the form $\mathrm{Div}(F)$ for certain rational function $F$ on $\tmkf\otimes_\bq \bbq$. Therefore $\mathrm{Div}(F|_{\mkf\otimes_\bq \bbq})=n\overline{Z}$ and $[Z]$ is equal to $0$ in $\CH^1_0(\tmkf\otimes_\bq \bbq)\otimes \bq$.

Now consider (ii). By the hypothesis, $[Z_1-Z_2]$ is a torsion element in $\CH^1(\mkf\otimes_\bq \bbq)$, whence there exists $n\in \bn$ so that $nZ_1-nZ_2$ is rationally equivalent to zero. Write $nZ_1-nZ_2=\mathrm{Div}(f)$ for certain rational function $f$ on $\mkf\otimes_\bq \bbq$. The divisor of $f$ on $\tmkf\otimes_\bq \bbq$ is of the shape $n\overline{Z}_1-n\overline{Z}_2+\sum_i n_i B_i$ with $n_i\in \bz$. So (ii) holds.
\end{proof}

We now pass to the direct limits $\CH^\ast(M\otimes_\bq \bbq)=\dlim \CH^\ast(\mkf\otimes_\bq \bbq)$, $\HH^\ast(M,\bc)=\dlim \HH^\ast(\mkf,\bc)$ and consider the according cycle map
\[
\cl_M:\CH^1(M\otimes_\bq \bbq)\otimes_\bq \bc\rar \HH^2(M,\bc).
\]
\begin{proposition}
$\cl_M$ is $G(\ba_f)$-equivariant and injective. $\HH^2(M,\bc)$ is a completely reducible admissible $G(\ba_f)$-module.
\end{proposition}
\begin{proof}
The $G(\ba_f)$-equivariance is obvioius. Because of Lemma ~\ref{cyclelemma} (i), the map $\cl_{\mkf}:\CH^1(\mkf\otimes_\bq \bbq)\otimes_\bq \bq\rar \HH^2(\mkf,\bq)$ is injective for neat $K_f$, whence $\cl_M$ is injective. By Section 13 and specifically Lemma 12 in \cite{Weiss92}, $\HH^2(M,\bc)$ is isomorphic to the $(\fg,K_\infty)$-cohomology of the discrete spectrum $L^2_{\disc}(G(\bq)Z(\br)^+\backslash G(\ba))$, whence $\HH^2(M,\bc)$ is a completely reducible admissible $G(\ba_f)$-module.
\end{proof}

We now prove Theorem ~\ref{main} by assuming Theorem ~\ref{picardgroup}.
\begin{proof}[Proof of Theorem ~\ref{main}]

(1) Theorem ~\ref{picardgroup} implies
\[
\CH^1(M\otimes_\bq \bbq)\otimes_\bz \bc=\SC^1(M)\otimes_\bz \bc.
\]
Hence $\CH^1(M\otimes_\bq \bbq)\otimes_\bz \bq=\SC^1(M)\otimes_\bz \bq$. Taking the $K_f$-invariant subspaces on both sides, we get $\CH^1(\mkf\otimes_\bq \bbq)\otimes_\bz \bq=\SC^1(\mkf)\otimes_\bz \bbq$.

(1)$\Longrightarrow$(2) By Theorem 9.4 in \cite{Weiss88}, the space of Tate classes $\Ta(\tmkf)$ is spanned by the image of $\CH^1(\tmkf\otimes_\bq \bbq)$. By (1) and Lemma ~\ref{cyclelemma} (ii), $\CH^1(\tmkf\otimes_\bq \bbq)\otimes_\bq \bbq$ is spanned by $[B_i]$ ($1\leq i\leq m$) and $[\overline{Z}]$ with $[Z]\in \SC^1(\mkf\otimes_\bq \bbq)$. The assertion in (2) then follows.

(2)$\Longrightarrow$(3) Choose a $\bq$-embedding $\tmkf\rar \bp^N$ for certain $N\in \bn$. Let $L_0$ be a $\bq$-hyperplane in $\bp^N$ that has a non-trivial intersection with $\tmkf$. Put $L=L_0\cap \tmkf$, then the class $\cl_{et}([L])\in H^2(\tmkf\otimes_\bq \bbq,\bq_\ell(1))$ is invariant by $\Gal(\bbq/\bq)$. By Hard Lefschetz Theorem, the map
\begin{align*}
 \ml: \HH^2(\tmkf\otimes_\bq \bbq,\bq_\ell(1)) &\lrar \HH^4(\tmkf\otimes_\bq \bbq,\bq_\ell(2))\\
 t&\lrar t\cup \cl_{et}([L]).
\end{align*}
is an isomorphism. It respects the action of $\Gal(\bbq/\bq)$ and hence, for every number filed $E$, induces an isomorphism
\[
\HH^4(\tmkf\otimes_\bq \bbq,\bq_\ell(2))^{\Gal(\bbq/E)}\cong \ml \left(\HH^2(\tmkf\otimes_\bq \bbq,\bq_\ell(1))^{\Gal(\bbq/E)}\right).
\]
Thus, $\Ta^2(\tmkf)=\ml(\Ta^1(\tmkf))$.

By (2), $\Ta(\tmkf)$ is spanned by the images of $[B_i]$ and $[\overline{Z}]$ for $1\leq i\leq m$ and $[Z]\in\SC^1(\mkf)$. Specifically, $\cl_{et}([L])$ is a linear combination of the images of these $[B_i]$ and special divisors. Therefore, $\Ta^2(\tmkf)=\ml(\Ta^1(\tmkf))$ is spanned by the listed elements in (3). (The map $\ml$ may depend on the embedding into $\bp^N$.)

(3)$\Longrightarrow$(4) Suppose $Z\in \mz^2(\mkf\otimes_\bq \bbq)$. By (3), there is a cycle class $[Z^\prime]$ in $\CH^1(\tmkf\otimes_\bq \bbq)\otimes_\bq$ that satisfies $\cl_{\tmkf,et}([\overline{Z}])=\cl_{\tmkf,et}([{Z^\prime}])$ and is of the shape $\sum_i q_i [\overline{Z}_i]+\sum_{i^\prime,j}q_{i^\prime,j}[\overline{Z}_{i^\prime}\cdot B_j]+\sum_{j_1,j_2}[B_{j_1}\cdot B_{j_2}]$ with $[Z_i], [Z_{i^\prime}]\in \SC^2(\mkf\otimes_\bq \bbq)$ and the coefficients in $\bq$. Because the cycle maps are compatible with the comparison maps between etale cohomology and singular cohomology, there is $\cl_{\tmkf}([\overline{Z}])=\cl_{\mkf}([{Z^\prime}])$ in $\HH^4(\tmkf,\bq)$. Thus
\[
\cl_{\mkf}([Z])=j\circ \cl_{\tmkf}([\overline{Z}])=j\circ \cl_{\tmkf}([{Z^\prime}])=\sum_i q_i\cl_{\mkf}([Z_i]).
\]
Here $j\circ \cl_{\tmkf}([\overline{Z}_{i^\prime}\cdot B_j])$ and $j\circ \cl_{\tmkf}([B_{j_1}\cdot B_{j_2}])$ vanish because the cycles are not supported on $\mkf(\bc)$. So $[Z]=\sum_i q_i [Z_j]$ in $\overline{\CH}^2(\mkf\otimes_\bq \bbq)$ and the assertion in (4) follows.
\end{proof}

\section{The decomposition of $\mathrm{H}^{1,1}(M,\bc)$}\label{repofh2}

For an irreducible unitary automorphic representation $\pi$ of $G(\ba)$, let $m(\pi)$ denote its multiplicity in $L^2_{\mathrm{disc}}(G(\bq)Z(\br)^+\backslash G(\ba))$. For an irreducible admissible unitary representation $\pi_f$ of $G(\ba_f)$, let $\HH^{1,1}(\pi_f)$ be the $\pi_f$-isotypic component of $\HH^{1,1}(M,\bc)$.

There is a decomposition
\begin{equation}\label{h11_d}
\mathrm{H}^{1,1}(M,\bc)=\oplus_{\pi}m(\pi)\mathrm{H}^{1,1}(\fg,K_\infty,\pi)=\oplus_{\pi_f}\mathrm{H}^{1,1}(\pi_f).
\end{equation}
Here $\mathrm{H}^{1,1}(\fg,K_\infty,\pi)=\mathrm{H}^{1,1}(\fg,K_\infty,\pi_\infty)\otimes \pi_f$. Recall that $\mathrm{H}^{1,1}(\fg,K_\infty,\pi_\infty)$ is $\bc$ when $\pi_\infty\in\{\pi^{2+},\pi^{2-}, 1,\sgn\circ\nu\}$ and zero otherwise. Similarly, one writes $\HH^2(M,\bc)=\HH^2(\pi_f)$ as the direct sum of isotypic components.

Weissauer determined that all $\pi$ occurring in (\ref{h11_d}) are the twists of three types of basic representations.
\begin{theorem}\cite[Thm. 4, Lemma 6]{Weiss92}\label{wei_coho}
If $\mathrm{H}^{1,1}(\fg,K_\infty,\pi)\neq 0$, then $\pi=\pi^\prime\otimes \chi$, where $\chi$ is a character with $\chi_\infty\in \{1,\sgn\circ \nu\}$ and $\pi^\prime$ is one of the following types:
\begin{itemize}
\item[(I)] a CAP representation of $\mathrm{PGSp}_4(\ba)$ of Siegel type $(\tau\boxtimes 1,\frac{1}{2})$ with $\tau\subset \ma_{cusp}(\PGL_2))$, $\pi_\infty=\pi^{2+}$ and $\tau_\infty=\fD_4$,

\item[(II)] $J(P,\tau,\frac{1}{2})$, with $\tau\subset \ma_{cusp}(\PGL_2)$, $\tau_\infty=\fD_4$, and $L(\frac{1}{2},\tau)\neq 0$. $J(P,\tau,\frac{1}{2})$ denotes the unique irreducible quotient of $\Pi(\tau\boxtimes 1,\frac{1}{2})$ and is realized as the residue representation of the Eisenstein series associated to $\Pi(\tau\boxtimes 1,z)$ at $z=\frac{1}{2}$.

\item[(III)] the trivial representation $1$.
\end{itemize}
\end{theorem}

Because representations of type I, II and III all have multiplicity one in the discrete spectrum, there is $m(\pi)=1$ when $\mathrm{H}^{1,1}(\fg,K_\infty,\pi)\neq 0$.

\begin{lemma}\label{multiplicity}
If $\HH^{1,1}(\pi_f)$ is nonzero, then there exists a unique $\pi_\infty$ such that $\pi=\pi_\infty\times \pi_f$ occurs in $L^2_{\mathrm{disc}}(G(\bq)Z(\br)^+\backslash G(\ba))$. As a consequence, $\HH^2(\pi_f)=\HH^{1,1}(\pi_f)$ is an irreducible admissible $G(\ba_f)$-module.
\end{lemma}
\begin{proof}
Since $\HH^{1,1}(\pi_f)$ is nonzero, there is $ \pi_\infty\in\{\pi^{2+},\pi^{2-}, 1,\sgn\circ\nu\}$ such that $\pi=\pi_\infty\times \pi_f$ occurs in the discrete spectrum. We show that if $\pi^\prime=\pi_\infty^\prime\times \pi_f$ occurs in the discrete spectrum, then $\pi^\prime=\pi$. By Theorem ~\ref{wei_coho}, one may assume that $\pi$ is one of the basic types.

(i) $\pi$ is of type I or II with respect to $\tau\subset \ma_{cusp}(\PGL_2)$. In this situation, there is an irreducible cuspidal $\wtilde{\SL}_2(\ba)$-representation $\sigma\in \Wd_\psi(\tau)$ such that $\pi=\Theta_{\wtilde{\SL}_2\times \PGSp_4}(\sigma,\psi)$.  When $\pi$ is of type I, $\Theta_{\wtilde{\SL}_2\times \PGL_2}(\sigma,\psi)=0$ (See Thm. ~\ref{ps_cap}); when $\pi$ is of type II, $\Theta_{\wtilde{\SL}_2\times \PGL_2}(\sigma,\psi)=\tau$ (See Lemma 10 in \cite{Weiss92}).

Note that for almost all finite $p$, $\pi^\prime_p=\pi_p$ is equal to $J(P,\tau_p,\frac{1}{2})$, the $p$-component of $J(P,\tau,\frac{1}{2})$. If $\pi^\prime$ is cuspidal, then it is CAP of Siegel type $(\tau\boxtimes 1,\frac{1}{2})$. If $\pi^\prime$ is in the discrete residue spectrum, then there must be $L(\frac{1}{2},\tau)\neq 0$ and $\pi^\prime=J(P,\tau,\frac{1}{2})$; this is because the residue spectrum of $\GSp_4(\ba)$ have been known explicitly and the fact $\pi_p^\prime=J(P,\tau_p,\frac{1}{2})$ for almost all $p$ selects out onely one possible choice $J(P,\tau,\frac{1}{2})$. In either case, there is $\pi^\prime=\Theta_{\wtilde{\SL}_2\times \PGSp_4}(\sigma^\prime,\psi)$ for certain $\sigma^\prime\in \Wd_\psi(\tau)$.

Because $\pi_f^\prime=\pi_f$, there is $\sigma_f^\prime=\sigma_f$. However, for two representations $\sigma$ and $\sigma^\prime$ in the same Waldspurger packet, their local  components can differ only at a even number of places. (See \cite[(1.8), (1.9)]{wgan2008} or \cite[Coro. 1, 2]{Wald91}). It forces that $\sigma_\infty^\prime=\sigma_\infty$, whence $\sigma^\prime=\sigma$ and $\pi^\prime=\pi$.

(ii) $\pi=1$ is of type III. So $\pi_f^\prime=1$ and this forces $\pi^\prime=1$. Actually, $G(\bq)Z(\br)^+\backslash G(\ba)/G(\ba_f)=G(\bq)Z(\br)^+\backslash G(\br)$ and any continuous function on it must be a constant function.

Since $m(\pi)=1$, $\HH^2(\pi_f)=\HH^2(\fg,K_\infty,\pi_\infty)\times \pi_f=\HH^{1,1}(\fg,K_\infty,\pi_\infty)\times \pi_f=\HH^{1,1}(\pi_f)$ is irreducible.
\end{proof}

\begin{corollary}\label{criterion0}
The map $\mathrm{cl}_M:\SC^1(M)\otimes_\bz \bc\lrar \mathrm{H}^{1,1}(M,\bc)$ is an isomorphism if and only if $\SC^1(\pi_f)$ is nonzero when $\HH^{1,1}(\pi_f)$ is nonzero.
\end{corollary}
\begin{proof}
Because $\mathrm{cl}_M:\SC^1(M,\bc) \lrar \mathrm{H}^{1,1}(M,\bc)$ is injective and $G(\ba_f)$-equivariant, it is an isomorphism if and only if $\cl(\pi_f):\SC^1(\pi_f)\rar \HH^{1,1}(\pi_f)$ is an isomorphism for each $\pi_f$ occurring in $\mathrm{H}^{1,1}(M,\bc)$. Since $\HH^{1,1}(\pi_f)$ is irreducible by Lemma ~\ref{multiplicity}, the injective homomorphism $\cl(\pi_f)$ is an isomorphism if and only if $\SC^1(\pi_f)\neq 0$.
\end{proof}

\section{The period pairing}\label{pairng}

A cycle $[Z]\in \SC^1(M,\bc)$ can be written as
\[
[Z]=\sum_{\pi_f} [Z]_{\pi_f},
\]
where $[Z]_{\pi_f}\in \SC^1(\pi_f)$ is called the $\pi_f$-component of $[Z]$. (Given $[Z]$, $[Z]_{\pi_f}$ is nonzero for only finitely many $\pi_f$ because $[Z]$ is invariant under certain compact open subgroup $K_f\subset G(\ba_f)$ and only finitely many $\pi_f$ occurring in $\HH^{1,1}(M,\bc)$ have a nontrivial $K_f$-invariant subspace.)

The natural way to test the image of a cycle in cohomology is to use the period pairing between cycles and cohomology with compact support. It is difficult to construct closed differential forms with compact support. So we consider the alternative of rapidly decreasing closed differential forms. There are two basic facts:
\begin{itemize}
\item[(i)] The pairing between $\HH^2(M,\bc)$ and $\HH^4_c(M,\bc)$ given by $<[\omega_1],[\omega_2]>$ $:=\int_M \omega_1\wedge \omega_2$ is perfect and $G(\ba_f)$-equivariant. The restricted pairing on the isotypic components
\[
\HH^2(\pi_f)\times \HH^4_c(\pi_f^\prime)\rar \bc
\]
is zero when $\pi_f^\prime\not\cong \pi_f^\vee$ and perfect when $\pi_f^\prime\cong \pi_f^\vee$. Here $\pi_f^\vee$ denotes the dual representation.

\item[(ii)] $\mathrm{H}^*_c(M,\bc)$ is defined using the cochain $\Omega^*_c(M,\bc)$ consisting of compactly supported differential forms on $M$. Let $\mathrm{H}^*_{rd}(M,\bc)$ be the cohomology groups defined using the cochain $\Omega^*_{rd}(M,\bc)$ consisting of rapidly decreasing differential forms on $M$. Borel \cite{borel1980} proved that the inclusion map $\Omega^*_c(M,\bc)\hrar \Omega^*_{rd}(M,\bc)$ induces an isomorphism $\Lambda:\mathrm{H}^*_c(M,\bc) \isoto \mathrm{H}^*_{rd}(M,\bc)$.
\end{itemize}

We make a simple but very useful observation below.
\begin{lemma}\label{criterion1}
Suppose $\HH^{1,1}(\pi_f)\neq 0$. If there exists $[Z]\in \SC^1(M)$ and a $\bk_f$-finite rapidly decreasing closed form $\Omega$ on $M$ such that \emph{(i)} $<\mathrm{H}^{1,1}(\pi^\prime_f),\Omega>=0$ for $\pi_f^\prime\neq \pi_f$ and \emph{(ii)} $\int_Z \Omega\neq 0$, then $[Z]_{\pi_f}$ is nonzero and as a consequence $\SC^1(\pi_f)$ is nonzero.
\end{lemma}
\begin{proof}
Let $[\Omega]$ be the cohomology class of $\Omega$ in $\HH^4_{rd}(M,\bc)$ and $\omega$ be a compactly supported closed form on $M$ that represents the class $\Lambda([\Omega])$ in $\HH^4_c(M,\bc)$, then $\Omega-\omega$ is the boundary of a rapidly decreasing form. Hence
\[
\int_Z \Omega=\int_Z \Omega-\omega+ \int_Z \omega=0+<\cl_M([Z]),[\omega]>.
\]
Write $[Z]=\sum_{\pi_f^\prime} [Z]_{\pi_f^\prime}$ and $[\omega]=\sum_{\pi_f^\pprime} [\omega]_{\pi_f^\pprime}$, where $[Z]_{\pi_f^\prime}$ is the $\pi_f^\prime$-component of $\pi_f^\prime$ and $[\omega]_{\pi_f^\pprime}$ is the $\pi_f^\pprime$-component of $[\omega]_{\pi_f^\pprime}$. These are finite sums because $[Z]$ and $[\Omega]$ (and hence $[\omega]$) are $\bk_f$-finite. Note that $\cl_M([Z]_{\pi_f^\prime})\in \HH^{1,1}(\pi_f^\prime)$.

By Condition (i), one has $<\HH^{1,1}(\pi_f^\prime),[\omega]>=0$ for $\pi_f^\prime\neq \pi_f$, whence
\[
<\cl_M([Z]),[\omega]>=\sum_{\pi_f^\prime}<\cl_M([Z]_{\pi_f^\prime}),[\omega]>=<\cl_M([Z]_{\pi_f}),[\omega]>.
\]
Since the pairing on $\HH^2(\pi_f)\times \HH^4(\pi_f^\pprime)$ is zero for $\pi_f^\pprime\neq \pi_f^\vee$, there is
\[
<\cl_M([Z]_{\pi_f}),[\omega]>=\sum_{\pi_f^\pprime}<\cl_M([Z]_{\pi_f}),[\omega]_{\pi_f^\pprime}>=<\cl_M([Z]_{\pi_f}),[\omega]_{\pi_f^\vee}>.
\]
So $\int_Z \Omega=<\cl_M([Z]_{\pi_f}),[\omega]_{\pi_f^\vee}>$. With Condition (ii), one sees that both $\cl_M([Z]_{\pi_f})$ and $[\omega]_{\pi_f^\vee}$ are nonzero. Specifically, $[Z]_{\pi_f}$ is a nonzero element in $\SC^1(\pi_f)$.
\end{proof}

So we propose the following proposition.
\begin{proposition}\label{formOmega}
If $\HH^{1,1}(\pi_f)$ is nonzero, then there exist $[Z]\in \SC^1(M)$ and a $\bk_f$-finite rapidly decreasing closed form $\Omega$ on $M$ such that \emph{(i)} $<\mathrm{H}^{1,1}(\pi_f^\prime),\Omega>=0$ for $\pi_f^\prime\neq \pi_f$, \emph{(ii)} $\int_Z \Omega \neq 0$.
\end{proposition}

\begin{lemma}\label{reduction}
Suppose that $\HH^{1,1}(\pi_f)$ is nonzero. Let $\pi_\infty$ be the unqiue member of $\{\pi^{2+},\pi^{2-}, 1, \sgn\circ \nu\}$ such that $\pi_\infty\times \pi_f$ occurs in the discrete spectrum. Let $\chi$ be a character of $\bq^\times\backslash \bq^\times$ with $\chi_\infty\in \{1,\sgn\}$. If Proposition ~\ref{formOmega} holds for $\pi_f$, then it also holds for $\pi_f\otimes \chi_f$.
\end{lemma}
\begin{proof}
Suppose that Proposition ~\ref{formOmega} holds for $\pi_f$, then there exist a $\bk_f$-finite rapidly decreasing closed form $\Omega$ on $M$ and $[Z]\in \SC^1(M)$ such that $<\mathrm{H}^{1,1}(\pi_f^\prime),\Omega>=0$ for $\pi_f^\prime\neq \pi_f$ and $\int_Z \Omega \neq 0$.

Choose a neat compact open subgroup $K_f$ sufficiently small so that (i) $\chi_f\circ \nu$ is invariant by $K_f$, (ii) $[Z]=\sum_i c_i[Z_i]$, where $Z_i$ are irreducible special diviors on $\mkf\otimes_\bq \bar{\bq}$. For each $i$, select an element $g_i$ in $G(\ba)$ that represents a point on $Z_i\subset \mkf$. Define the $\chi$-twist of $Z$ by $Z_{\chi}:=\sum_i c_i \chi^{-1}(\nu(g_i)) Z_i$. Then $\Omega_\chi:=\chi(\nu(g))\Omega^\prime$ and $Z_{\chi}$ satisfy the requirement of Proposition ~\ref{formOmega} for $\pi_f\otimes \chi_f$.
\end{proof}

By Corollary ~\ref{criterion0} and Lemma ~\ref{criterion1}, Theorem ~\ref{picardgroup} is a consequence of Proposition ~\ref{formOmega}. Furthermore, by Theorem ~\ref{wei_coho} and Lemma ~\ref{reduction}, it is sufficient to verify Proposition ~\ref{formOmega} when $\pi=\pi_\infty\times \pi_f$ is of basic type I, II, and III in Theorem ~\ref{wei_coho}. In next three section, we prove Proposition ~\ref{formOmega} when $\pi$ is one of the three basic types I, II, and III. This would complete the proof of Theorem ~\ref{picardgroup}. Note that $\pi$ is self-dual in these cases.

\section{Nonvashing Periods I}\label{np1}

We verify Proposition ~\ref{formOmega} when $\pi$ is of type I. The candidates for $[Z]$ are non-split cycles in $\SC^1_{ns}(M)$ and the candidates for $\Omega$ are forms in $\mathrm{H}^{2,2}(\fg,K_\infty,\pi)$, which are cuspidal and hence rapidly decreasing.

\begin{proposition}\label{cuspidalcase}
Let $\pi$ be a CAP representation of $\mathrm{PGSp}_4(\ba)$ of Siegel type $(\tau,\frac{1}{2})$ with $\pi_\infty=\pi^{2+}$. There exist $[Z]\in \SC^1_{ns}(M)$ and $\Omega\in \HH^{2,2}(\fg,K,\pi)$ such that \emph{(i)} $<\mathrm{H}^{1,1}(\pi_f^\prime),\Omega>=0$ for all $\pi_f^\prime\neq \pi_f$, \emph{(ii)} $\int_Z \Omega \neq 0$.
\end{proposition}

Proposition ~\ref{cuspidalcase} has an immediate corollary below. We prove Proposition ~\ref{cuspidalcase} at the end of this section, after preparing relevant lemmas.

\begin{corollary}
Let $\mathrm{H}^{1,1}_{cusp}(M,\bc)$ be the subspace of $\mathrm{H}^{1,1}(M,\bc)$ spanned by cuspidal closed differential forms, then $\mathrm{H}^{1,1}_{cusp}(M,\bc)$ is spanned by the images of certain cycle classes in $\SC^1_{ns}(M)$.
\end{corollary}

\subsection{Nonvanishing of automorphic periods}

Recall the identification $\GSp_4\cong \GSpin(V)$ in Section ~\ref{group}. For a nonisotropic vector $v\in V$. Let $\pi$ be an irreducible cuspidal unitary representation in $L^2_{\disc}(G(\bq)Z(\br)^+\backslash G(\ba))$. For a nonisotropic vector $v\in V$, define a period functional $\mpp_{v}$ on the space of smooth vectors in $\pi$,
\[
\mpp_{v}(\varphi)=\int_{[\SO({v}^\perp]} \varphi(h)dh.
\]

\begin{lemma}\label{nonvanishing}
Let $\pi$ be a CAP representations of $\mathrm{PGSp}_4(\ba)\cong\mathrm{SO}(V_\ba)$ of Siegel type $(\tau,\frac{1}{2})$ with $\pi_\infty=\pi^{2+}$. There exists a vector $v\in V$ with $q(v)\in \bq_+\backslash {\bq^\times}^2$ such that $\mpp_{v}$ is nonzero.
\end{lemma}
\begin{proof}
Choose a non-trivial character $\psi$ of $\ba/\bq$. By Theorem ~\ref{ps_cap}, $\pi=\Theta_{\wtilde{\SL}_2\times \PGSp_4}(\sigma, \psi)$ with $\sigma\in \Wd_\psi(\tau)$ satisfying $\Theta_{\wtilde{\SL}_2\times \PGL_2}(\sigma,\psi)= 0$. There exists $a\in \bq$ such that the $\psi_a$-Whittaker functional $\ell_{\psi_a}$ is nonzero on $\sigma$. This implies $\Theta_{\wtilde{\SL}_2\times \PGL_2}(\sigma,\psi_a)=\tau\otimes \chi_a$, whence $a\not\in {\bq^\times}^2$. Because $\pi_\infty=\pi^{2+}$, we must have $a\in \bq_+$. Otherwise, $\sigma_\infty=\theta_{\wtilde{\SL}_2\times \PGL_2}(\fD_4\otimes \chi_{a_\infty},\psi_{a_\infty})=\theta_{\wtilde{\SL}_2\times \PGL_2}(\fD_4,\psi^{-1}_\infty)$ and, by Lemma ~\ref{localtheta}, $\pi^{2+}=\theta_{\wtilde{\SL}_2\times \PGSp_4}(\sigma_\infty, \psi_\infty)$ is a discrete series, which is a contradiction.

We have $\sigma=\Theta_{\wtilde{\SL}_2\times \PGSp_4}(\pi,\psi)$. For a form $\varphi=\Theta(\phi,f)$ with $f\in \pi$ and $\phi\in \ms(V(\ba))$, there is
\begin{align*}
\ell_{\psi_a}(\varphi)=&\int_{\bq\backslash \ba}\big[\int_{[\SO(V)]}\overline{f(h)}\sum_{\xi\in V}\omega(\smalltwomatrix{1}{n}{}{1},h)\phi(\xi)dh\big] \psi(-an) dn\\
=&\int_{[\SO(V))]} \overline{f(h)} \sum_{\xi\in V} \omega(h)\phi(\xi) \int_{\bq\backslash \ba}\psi\big((q(\xi)-a)n)dn.
\end{align*}
The integral $\int_{\bq\backslash \ba}\psi\big((q(\xi)-a)n)dn$ is zero when $q(\xi)\neq a$ and $1$ when $q(\xi)=a$. Let $v$ be a vector in $V$ with $q(v)=a$, then vectors $\xi$ with $q(\xi)=a$ can be expressed as $\gamma\cdot v$, with $\gamma\in \SO(v^\perp)_\bq\backslash \SO(V)_\bq$. So,
\begin{align*}
\ell_{\psi_a}(\varphi)=&\int_{[\SO(V)]}\sum_{\gamma\in \SO(v^\perp)_\bq\backslash \SO(V)_\bq} \overline{f( h)}\phi(h^{-1}\gamma^{-1} v)dh\\
=&\int_{\SO(v^\perp)_\ba\backslash \SO(V)_\ba} \overline{f(h)}\phi(h^{-1}\circ v)dh\\
=&\int_{\SO(v^\perp)_\ba\backslash \SO(V)(\ba)}\big(\int_{[\SO(v^\perp)]}f(hh^\prime)dh\big)\phi({h^\prime}^{-1}v)dh^\prime.
\end{align*}
Because $\ell_{\psi_a}$ is nonzero on $\sigma$, the function $\int_{[\SO(v^\perp)]}f(hh^\prime)dh=\mpp_{v}(\pi(h^\prime)f)$ is not identically zero. Therefore, $\mpp_{v}$ is nonzero.
\end{proof}

\subsection{Periods of cohomological forms}\label{pocf}

We describe the periods of cohomological forms in $\mathrm{H}^4(\fg,K_\infty,\pi)$ on $Z_{\bq v,g_f,K_f}$. Note the following:
\begin{itemize}
\item[(i)] The homomorphism $\GSp_4(\br) \isoto \SO(V)(\br)$ induces an isomorphism $\Lie(\mathrm{Sp}_4(\br))\isoto \Lie(\SO(V)_\br)$. We take the Cartan decomposition $\fg_0=\fp\oplus \fk$ in Section ~\ref{algebra} as the Cartan decomposition of $\Lie(\SO(V)_\br)\otimes \bc$.

\item[(ii)] Recall the compact torus $T$ whose Lie algebra is $\ft_\br$. There exists $g_\infty\in G(\br)$ such that $g_\infty T g_\infty^{-1}$ is a maximal connected compact subgroup of $\GSpin({v}^\perp)_\br$. With this choice of $g_\infty$, $\fp^\prime:=\fp\cap \Ad_{g_\infty^{-1}}\big[\Lie(\SO(V)_\br)\otimes \bc\big]$ is 4-dimensional and $\ft$-invariant. Thus $\Ad_{g_\infty} \fp^\prime\oplus \Ad_{g_\infty} \ft$ is a Cartan decomposition of $\Lie(\SO(V)_\br)\otimes \bc$.
\end{itemize}

Recall that the $\fk$-submodule $\pi^{2+}(\delta_{2\alpha})$ of $\pi^{2+}$ has a basis $\{v_j, -2\leq j\leq 2\}$ consisting of weight vectors (see Sect. ~\ref{coho_liealg}). The vector $v_0$ of weight $0$ is particularly important, as shown by the lemma below.

\begin{lemma}\label{formulation}\label{period-transfer}
Let $\pi$ be a CAP representations of $\mathrm{PGSp}_4(\ba)$ of Siegel type $(\tau,\frac{1}{2})$ with $\pi_\infty=\pi^{2+}$. Let $v$ be as in Lemma ~\ref{nonvanishing} and choose $g_\infty\in G(\br)$ such that $g_\infty T g_\infty^{-1}$ is a maximal connected compact subgroup of $\GSpin({v}^\perp)_\br$. For the special divisor $Z_{\bq v,g_f,K_f}$, there is
\[
\{\int_{Z_{\bq v,g_f,K_f}}\omega: \omega\in \mathrm{H}^4(\fg,K_\infty,\pi)\}=\{\mpp_v(\pi(g_\infty)\varphi)| \varphi \in v_0\otimes \pi_f\}.
\]
\end{lemma}
\begin{proof}
Choose a basis $\{X_i\}_{i=1}^4$ of $\fp^\prime$ consisting of weight vectors with respect to $\ft$ and add two other weight vectors $X_5,X_6$ to form a basis of $\fp$. Let $\{\omega_i\}_{i=1}^6$ be the dual basis in $\fp^*$. For a subset $I=\{i_1<\cdots<i_q\}$ of $\{1,\cdots,6\}$, set $\omega_I=\omega_{i_1}\wedge\cdots\wedge \omega_{i_q}$. By the discussion in Section ~\ref{coho_liealg} about the $\fk$-module structure of $\pi^{2+}$ and $\wedge^4 \fp\cong \wedge^4 \fb_0$, there is
\[
\mathrm{H}^4(\fg,K_\infty,\pi^{2+})= \Hom_{K_\infty}(\wedge^4\fp,\pi^{2+})=\bc\cdot \big(\sum_{|I|=4} v_I\cdot \omega_I\big),
\]
where $v_I\in \pi^{2+}(\delta_{2\alpha})$ and its weight is the negative of the weight of $\omega_I$.

Thus, a form in $\HH^4(\fg,K_\infty,\pi)$ is of the shape $\omega=\sum_{|I|=4} \varphi_I\omega_I$, with $\varphi_I=v_I\otimes v_f\in \pi$ for certain $v_f\in \pi_f$. Let $R$ denote the right translation action on $G(\ba)$. There is
\begin{align*}
\int_{Z_{\bq v,g_f,K_f}}\omega
=&\int_{\GSpin({v}^\perp)_\bq \backslash
\GSpin({v}^\perp)_\ba/L_\infty}
 R_{g_\infty}^*\omega \\
=&c\int_{\mathrm{SO}(v^\perp)_\bq \backslash
\mathrm{SO}(v^\perp)_\ba/\bar{L}_\infty}
 \sum_{|I|=4}
 \varphi_I(gg_\infty)\mathrm{Ad}^*_{g_\infty^{-1}}(\omega_I).
\end{align*}
Here $L_\infty=g_\infty K_\infty g_\infty^{-1}$, $\overline{L}_\infty=L_\infty/Z(\br)^+$, and $c$ is as in Section ~\ref{sect_notation}.

One needs to determine the restriction of $\mathrm{Ad}^*_{g_\infty^{-1}}(\omega_I)$ to $\SO(v^\perp)_\br/\overline{L}_\infty$. Because $\Ad_{g_\infty}\fp^\prime\oplus \Ad_{g_\infty}\ft$ is a Cartan decomposition of $\Lie(\SO(v^\perp)_\br)\otimes \bc$ and $\Ad_{g_\infty}\ft=\Lie(\overline{L}_\infty)\otimes \bc$, left invariant vector fields on $\SO(v^\perp)_\br/\overline{L}_\infty$ are identified with elements in $\Ad_{g_\infty}\fp^\prime$. Hence the restriction of $\Ad^\ast_{g_\infty^{-1}}\omega_i$ to $\SO(v^\perp)_\br/\overline{L}_\infty$ is nonzero if and only if $i=1,2,3,4$. Therefore, among all $I$ with $|I|=4$, only $\mathrm{Ad}^*_{g_\infty^{-1}}(\omega_{1234})$ has nonzero restriction. Its restriction, when combined with the volume form of $\overline{L}_\infty$, gives the volume form of $\SO(v^\perp)_\br$. So
\[
\int_{Z_{\bq v,g_f,K_f}}\omega
=c \int_{SO(v^\perp)_\bq \backslash \SO(v^\perp)_\ba}\varphi_{1234}(gg_\infty)dg
=c \mpp_v(\pi(g_\infty)\varphi_{1234}).
\]
Note that $\varphi_{1234}$ belongs to $v_0\otimes \pi_f$ because it is of weight $0$.

On the other hand, given $\varphi=v_0\otimes v_f \in v_0\otimes \pi_f$, set $\omega=\sum_{|I|=4} \varphi_I \omega_I$ with $\varphi_I=v_I\otimes v_f$, then $\mpp_v(\pi(g_\infty)\varphi)=c^{-1}\int_{Z_{\bq v,g_f,K_f}}\omega$. So we obtain the equality of sets in the Lemma.
\end{proof}

\subsection{Period relation}

We show that $\mpp_v$ is nonzero on $\pi$ if and only if it is nonzero on the subspace $v_0\otimes \pi_f$. We prove this assertion by arguing that the value of $\mpp_v$ on a general vector is a scalar multiple of its value on a vector in $v_0\times \pi_f$. Such a relation exists partly because the $K_\br$-type to which $v_0$ belongs is minimal in $\pi^{2+}$.

\begin{lemma}\label{HLie}
Let $\pi$ be an irreducible cuspidal automorphic representation of $\SO(V)(\ba)$ and $v$ be a nonisotropic vector of $V$. For a smooth vector $\varphi\in \pi$ and $X\in \Lie(\SO(V)_\br)$, there is $\mpp_v(X\circ \varphi)=0$.
\end{lemma}
\begin{proof}
Write $F(h,t)=\frac{\varphi(he^{tX})-\varphi(h)}{t}$, then $X\circ \varphi(h)=\underset{t\rar 0}{\lim}F(h,t)$ and 
\[
\mpp_v(X\circ \varphi)=\int_{[\SO(v^\perp)]} \lim_{t\rar 0} F(h,t)dh.
\]
By Mean Value Theorem, for each couple $(h, t)$, there exists some number $\xi_{h,t}$ between $0$ and $t$ such that $F(h,t)=X\circ \varphi(he^{\xi_{h,t}X})$. Because $\pi$ is cuspidal, both $\varphi$ and $X\circ \varphi$ are rapidly decreasing. So there exists an integrable function $\wtilde{F}(h)$ on $\SO(V)_\bq\backslash \SO(v^\perp)_\ba$ such that $|F(h,t)|\leq \wtilde{F}(h)$ when $|t|$ is small. By the dominated convergence theorem, there is
\[
\mpp_v(X\circ \varphi)=\lim_{t\rar 0} \int_{[\SO(v^\perp)]} F(h,t)dh=\lim_{t\rar 0} \frac{1}{t}\big[\mpp_v\big(\pi(e^{tX})\varphi\big)-\mpp_v(\varphi)\big].
\]
Because $X\in \mathrm{Lie}(\SO(v^\perp)_\br)$, there is $\mpp_v\big(\pi(e^{tX})\varphi\big)=\mpp_v(\varphi)$, whence $\mpp_v(X\circ \varphi)=0$.
\end{proof}

\begin{lemma}\label{specialform}
Let $\pi$ be a CAP representations of $\mathrm{PGSp}_4(\ba)$ of Siegel type $(\tau,\frac{1}{2})$ with $\pi_\infty=\pi^{2+}$. Let $v$ and $g_\infty$ be as in Lemma ~\ref{nonvanishing} and ~\ref{period-transfer}. For a $K_\infty$-finite smooth vector $\varphi=v_\infty\otimes v_f\in \pi$, put $\varphi_0=v_0\otimes v_f$, then $\mpp_v(\pi(g_\infty)\varphi)=C\mpp_v(\pi(g_\infty)\varphi_0)$ for a number $C$ depending on $v_\infty$.
\end{lemma}
\begin{proof}
Let $\muu$ denote the universal enveloping algebra of $\fg_0$. Because $v_\infty$ is $K_\infty$-finite, $v_\infty=R\cdot v_0$ for certain $R\in \muu$. Put $\ell:=\mpp_v\circ \pi(g_\infty)$ and observe the following facts:
\begin{itemize}
\item[(i)] By the choice of $g_\infty$, there is $\mathrm{Lie}(\SO(v^\perp)_\br)\otimes \bc=\Ad_{g_\infty}\fp^\prime\oplus \Ad_{g_\infty}\ft$ with $\fp^\prime\subset \fp$. Because $\fp^\prime\oplus \ft$ is closed under Lie brackets, there is
\[
\fp^\prime=V_{\alpha+\beta}\oplus V_{-\alpha-\beta}\oplus V_{\alpha-\beta}\oplus V_{-\alpha+\beta}.
\]
Set $\fh:=\ft\oplus V_{\alpha+\beta}\oplus V_{-\alpha-\beta}\oplus V_{\alpha-\beta}\oplus V_{-\alpha+\beta}$. By Lemma ~\ref{HLie}, $\ell\circ X=0$ for $X\in \fh$.
\item[(ii)]The Casimir element $\Omega$ in $\muu$ acts by a scalar, say $\lambda$, on the smooth vectors of $\pi^{2+}$. Choose a nonzero vector $E_\gamma\in V_\gamma$ for each positive root $\gamma$ of $\fg_0$. By choosing a suitable nonzero vector $E_{-\gamma}\in V_{-\gamma}$ and setting $H_\alpha=[E_\alpha,E_{-\alpha}]$, $H_\beta=[E_\beta,E_{-\beta}]$, one can write $\Omega$ as
\[
 \Omega=F(H_\alpha,H_\beta)+E_\alpha E_{-\alpha}+E_\beta E_{-\beta}+ E_{\alpha+\beta} E_{-\alpha-\beta}+E_{\alpha-\beta}E_{-\alpha+\beta},
\]
where $F(\cdot,\cdot)$ is certain degree-2 polynomial. By the observation made in (i), there is
\[
\ell\circ E_\beta E_{-\beta}=\lambda \ell-\ell\circ E_\alpha E_{-\alpha}.
\]
\item[(iii)]By Poincar\'{e}-Birkhoff-Witt theorem, we may write $R$ as
\[
R=R^\prime+\underset{0\leq i_1,i_2,j_1,j_2\leq n}{\sum}c_{i_1,i_2,j_1,j_2}E_\beta^{i_1}E_{-\beta}^{i_2}E_\alpha^{j_1} E_{-\alpha}^{j_2},
\]
where $R^\prime\in \fh\cdot \muu$ and $n\in \bn$.

\item[(iv)] $\ell$ vanishes on vectors of nonzero weight. Suppose that $\varphi^\prime$ is of weight $\gamma\neq 0$; choose $X\in \ft$ satisfying $\gamma(X)\neq 0$, then $0=\ell(X\varphi^\prime)=\ell(\gamma(X)\varphi^\prime)=\gamma(X)\ell(\varphi^\prime)$, whence $\ell(\varphi^\prime)=0$.
\end{itemize}

Now for $\varphi_0$, by (i) and (iii), there is
\[
\ell(R\varphi_0)=\underset{0\leq i_1,i_2,j_1,j_2\leq n}{\sum}c_{i_1,i_2,j_1,j_2}\ell[E_\beta^{i_1}E_{-\beta}^{i_2}E_\alpha^{j_1} E_{-\alpha}^{j_2}\varphi_0].
\]
Because $\varphi_0$ is of weight zero, the vector $E_\beta^{i_1}E_{-\beta}^{i_2}E_\alpha^{j_1} E_{-\alpha}^{j_2}\varphi_0$ is of weight $(i_1-i_2)\beta+(j_1-j_2)\alpha$. Applying (iv), we get that
\[
\ell(R\varphi_0)=\underset{0\leq i,i,j,j\leq n}{\sum}c_{i,i,j,j}\ell[E_\beta^{i}E_{-\beta}^{i}E_\alpha^{j} E_{-\alpha}^{j}\varphi_0].
\]
Because $E_\alpha$ and $E_{-\alpha}\in \fk$, the vector $E_\alpha^{j} E_{-\alpha}^{j}\varphi_0$ is a scalar multiple of $v_0$. So there are number $c_i$ ($0\leq i\leq n$) such that 
\[
\ell(R\varphi_0)=\underset{0\leq i\leq n}{\sum} c_i\ell(E_\beta^i E_{-\beta}^i \varphi_0).
\]

Furthermore, the relation $E_\beta E_{-\beta}=E_{-\beta} E_\beta+H_\beta$ implies $E_\beta^j H_\beta =\big(H_\beta-j\beta(H_\beta)\big)E_\beta^j$ and
\begin{align*}
&E_\beta^i E_{-\beta}^i\\
=&E_\beta E_{-\beta}E^{i-1}_\beta E^{i-1}_{-\beta}+\sum_{1\leq j\leq i-1} E_\beta^j H_\beta E_\beta^{i-1-j} E_{-\beta}^{i-1}\\
=&E_\beta E_{-\beta}E^{i-1}_\beta E^{i-1}_{-\beta}+\sum_{1\leq j\leq i-1} \big(H_\beta-j\beta(H_\beta)\big)E_\beta^{i-1} E_{-\beta}^{i-1}\\
=&E_\beta E_{-\beta}E^{i-1}_\beta E^{i-1}_{-\beta}+(i-1)H_\beta E_\beta^{i-1} E_{-\beta}^{i-1}-\frac{i(i-1)\beta(H_\beta)}{2}E_\beta^{i-1} E_{-\beta}^{i-1}
\end{align*}
Applying (i) and (ii), one gets 
\begin{align*}
&\ell(E_\beta^{i} E_{-\beta}^{i}\varphi_0)\\
=&\ell(E_\beta E_{-\beta}E^{i-1}_\beta E^{i-1}_{-\beta}\varphi_0)-\frac{i(i-1)\beta(H_\beta)}{2}\ell(E_\beta^{i-1} E_{-\beta}^{i-1}\varphi_0)\\
=&\big(\lambda-\frac{i(i-1)\beta(H_\beta)}{2}\big)\ell(E_\beta^{i-1} E_{-\beta}^{i-1}\varphi_0)
-\ell(E_{\alpha}E_{-\alpha}E_\beta^{i-1} E_{-\beta}^{i-1}\varphi_0)\\
=&\big(\lambda-\frac{i(i-1)\beta(H_\beta)}{2}\big)\ell(E_\beta^{i-1} E_{-\beta}^{i-1}\varphi_0)
-\ell(E_\beta^{i-1} E_{-\beta}^{i-1}E_{\alpha}E_{-\alpha}\varphi_0)\\
&-\ell([E_{\alpha}E_{-\alpha},E_\beta^{i-1} E_{-\beta}^{i-1}]\varphi_0).
\end{align*}
Note that the second term is a multiple of $\ell[E_\beta^{i-1} E_{-\beta}^{i-1}\varphi_0]$ and that, by (i) and PBW theorem, the third term is a linear combination of $\ell[E_\beta^{i^\prime-1} E_{-\beta}^{i^\prime-1}\varphi_0] (0\leq i^\prime\leq i-1)$. By induction, we have $\ell[E_\beta^{i} E_{-\beta}^{i}\varphi_0]=C_i\ell(\varphi_0)$. (It obviously holds when $i=0$.) Therefore, $\ell(R\varphi_0)=C\ell(\varphi_0)$ for $C=\sum_{0\leq i\leq n}c_i C_i$.
\end{proof}

\begin{proof}[Proof of Proposition ~\ref{cuspidalcase}]

We first apply Lemma ~\ref{nonvanishing} to get a vector $v\in V$ with $q(v)\in \bq^+\backslash {\bq^\times}^2$ such that $\mpp_v$ is nonzero on the space of smooth vectors of $\pi$. Choose a $K_\infty$-finite smooth vector $\varphi=v_\infty\otimes v_f$ of $\pi$ such that $\mpp_v(\varphi)\neq 0$. (This is possible because $K_\infty$-finite smooth decomposable vectors span a dense subspace in the space of smooth vectors and $\mpp_v$ is continuous.) By Lemma ~\ref{specialform}, $\mpp_v(\varphi_0)\neq 0$ for $\varphi_0=v_0\otimes v_f\in v_0\times \pi_f$. By Lemma ~\ref{period-transfer}, there exists $\Omega\in \HH^{2,2}(\fg,K_\infty,\pi)=\HH^{2,2}(\pi_f)$ such that $\int_{Z_{\bq v,g_f,K_\infty}}\Omega\neq 0$. Condition (i) in Proposition ~\ref{cuspidalcase} obviously holds: for $<\HH^{1,1}(\pi_f^\prime),\HH^{2,2}(\pi_f)>\neq 0$, it is necessary that $\pi_f^\prime=\pi_f^\vee=\pi_f$.
\end{proof}

\section{Nonvanishing Periods II}\label{np2}

We verify Proposition ~\ref{formOmega} when $\pi$ is of type II. The candidates for $[Z]$ are cycles in $\SC^1_{s}(M)$ (see Sect. ~\ref{cycle}) and the form $\Omega$ will be constructed using Harder's method of Eisenstein cohomology.

It is more convenient to describe $\SC^1_{s}(M)$ in terms of the group
\[
H:=\{(g_1,g_2):g_1,g_2\in \GL_2, \det g_1=\det g_2\}.
\]
and the embedding $H\hrar \GSp_4$ given by
\[
 \left(\smalltwomatrix{a_1}{b_1}{c_1}{d_1},
 \smalltwomatrix{a_2}{b_2}{c_2}{d_2}\right)\lrar
 \left(\begin{smallmatrix}
 a_1 & &b_1 &\\
 &a_2& &b_2\\
 c_1 & &d_1&\\
 &c_2 & &d_2
 \end{smallmatrix}\right).
\]
Note that $\bk_{H}:=\bk\cap H(\ba)=\prod_p \bk_{H,p}$ is a maximal compact subgroup of $H(\ba)$, where $\bk_{H,p}=\bk_p\cap H(\bq_p)$. Set $K_{H,\infty}=H(\br)\cap K_\infty$ and $K_{H,f}=H(\ba_f)\cap K_f$ for a compact open subgroup $K_f$ of $G(\ba_f)$. Let $Z_{H,K_f}$ denote the image of the Shimura variety $H(\bq)\backslash H(\ba)/K_{H,\infty}K_{H,f}$ in $\mkf$. The connected components of $Z_{H,K_f}$ and their Hecke translates, with $K_f$ varying, span $\SC^1_s(M)$.

\begin{proposition}\label{eseriescase}
Let $\pi=J(P,\tau,\frac{1}{2})$ with $\tau\subset \ma_{cusp}(\PGL_2)$, $\tau_\infty=\fD_4$ and $L(\frac{1}{2},\tau)\neq 0$. There exists a $\bk_f$-finite rapidly decreasing closed form $\Omega$ on M and a sufficiently small $K_f$ such that \emph{(i)} $<\mathrm{H}^{1,1}(\pi_f^\prime),\Omega>=0$ for $\pi_f^\prime\neq \pi_f$, \emph{(ii)} $\int_{Z_{H,K_f}}\Omega \neq 0$.
\end{proposition}
\begin{proof}
We will construct a differential form $E(\eta^{\kappa,\omega})$ in Section ~\ref{ss_ecoho}. Proposition ~\ref{eseriescase} is a consequence of Lemma ~\ref{closedness}, ~\ref{vp_p} and ~\ref{nvp_p}.
\end{proof}

\begin{corollary}\label{coro_eisen2}
Let $\pi$ be as in Proposition ~\ref{eseriescase} and $\chi$ be a character of $\bq^\times\backslash \ba$ with $\chi_\infty\in \{1,\sgn\}$, then $\mathrm{H}^{1,1}(\pi_f\otimes \chi_f)$ is spanned by the images of certain split special divisors.
\end{corollary}

\subsection{Eisenstein cohomology}\label{ss_ecoho}

We construct $\Omega$ by the method of Eisenstein cohomology, which was first used by Harder in \cite{harder1975}. The key observation is that $\HH^3(\fg,K_\infty,\Pi(\tau,\frac{1}{2}))$ is nonzero when $\tau_\infty=\fD_4$. We wedge a form in this space with a degree-1 form on $P(\bq)\backslash G(\ba)/K_\infty$. The wedge product is a closed form on $P(\bq)\backslash G(\ba)/K_\infty$ and we then sum its translates by $P(\bq)\backslash G(\bq)$ to obtain a form on $G(\bq)\backslash G(\ba)/K_\infty$.

Fix a Levi decomposition $P=UM$ as below,
\begin{align*}
&U:=\left\{u=\smalltwomatrix{I_2}{n}{}{I_2}: n\in \mathrm{Sym}_{2\times 2} \right\},\\
&M:=\left\{m=\smalltwomatrix{A}{}{}{x\leftup{t}{A^{-1}}}: A\in
\mathrm{GL}_2, x\in \mathrm{GL}_1 \right\}.
\end{align*}
Set $\fm=\Lie(M(\br))\otimes \bc$, $\fu=\Lie(U(\br))\otimes \bc$, $K_{M,\infty}=K_\infty\cap M(\br)$ and $\tilde{\fk}_M=\Lie(K_{M,\infty})\otimes_\br \bc$. Define a quasi-character $\lambda$ on $P(\bq)\backslash P(\ba)$,
\begin{equation*}
\lambda(p)=\big|\frac{\det A}{x}\big|,\quad p=\smalltwomatrix{A}{\ast}{}{x\leftup{t}{A^{-1}}}.
\end{equation*}

Set $P_H=H\cap P$, $M_H=H\cap M$, $U_H=H\cap U$.

\subsubsection{The $(\fg,K_\infty)$-cohomology}

By \cite[Sect. 3.3, 3.4]{bw2000}), there is
\begin{align}\label{cohop0}
\nonumber &\HH^3(\fg,K_\infty,I_{P,\fD_4\otimes 1,\frac{1}{2}})\\
\nonumber =&\HH^3(\fp, K_{M,\infty}, \fD_4\otimes \lambda^2)\\
\nonumber =&\HH^1(\fm,K_{M,\infty},\fD_4\otimes \HH^2(\fu,\bc)\otimes \lambda^2)\\
=&\Hom_{K_{M,\infty}}(\fm/\tilde{\fk}_M,\fD_4\otimes \HH^2(\fu,\bc)\otimes \lambda^2).
\end{align}
Note that $\ft\cap \wtilde{\fk}_M=\bc H$ and that $\fm/\wtilde{\fk}_M=\bc h\oplus \bc n_0$ with respect to the identification $\fg/\wtilde{\fk}=\fb_0$. Observe the following:
\begin{itemize}
\item[(i)] $(\fm/\tilde{\fk}_M)^*=\bc \eta^+\oplus\bc\eta^-$, where $\eta^\pm=h^*\pm\frac{i}{2}n_0^*$ are of weights $\pm\alpha$ with respect to $\ft\cap \wtilde{\fk}_M$.
\item[(ii)] $\mathrm{H}^2(\fu,\bc)=\bc\eta_+\oplus\bc \eta_-$, where $\eta_\pm=n_1^*\wedge n_3^*\pm in_1^*\wedge n_2^*$ are of weights $\pm\alpha$ with respect to $\ft\cap\wtilde{\fk}$.
\item[(iii)] $\fD_4=\oplus_{n\in \bz, |n|\geq 2} \fD_4(n\alpha)$, where $\fD_4(n\alpha)$ refers to the weight space of $\fD_4$ with weight $n\alpha$ and is $1$-dimensional.
\item[(iv)] $K_{M,\infty}\cong Z(\br)^+ \times \OO(2)_\br$. An element $(\alpha,A)$ on the RHS corresponds to $\smalltwomatrix{\alpha A}{}{}{\alpha\leftup{t}{A^{-1}}}$ on the LHS. The element $\smalltwomatrix{1}{}{}{-1}\in \OO(2)_\br$ sends $\eta^\pm$ to $\eta^\mp$, $\eta_\pm$ to $-\eta_\mp$, and $\fD_4(n\alpha)$ to $\fD_4(-n\alpha)$.
\end{itemize}
By choosing $v_+\in \fD_4(2\alpha)$ and setting $v_-=-\smalltwomatrix{1}{}{}{-1}\circ v_+$, we can write
\begin{align}\label{cohop}
\nonumber &\Hom_{K_{M,\infty}}(\fm/\tilde{\fk}_M,\fD_4\otimes \HH^2(\fu,\bc)\otimes \lambda^2).\\
=&\bc\cdot (v_+\eta^-\wedge\eta_-\otimes 1+v_-\eta^+\wedge\eta_+\otimes 1).
\end{align}

\subsubsection{The global induction}\label{gicoho}

By (\ref{cohop}), there is a global isomorphism
\begin{align}\label{cohop_global}
\nonumber &\HH^3(\fg,K_\infty,\Pi(\tau,\frac{1}{2}))\\
\cong &\HH^1(\fm,K_{M,\infty},\fD_4\otimes \HH^2(\fu,\bc)\otimes \lambda^2)\otimes \Pi(\tau_f,\frac{1}{2}).
\end{align}
Here $\Pi(\tau_f,\frac{1}{2})$ refers to the abstract induced representation of $G(\ba_f)$ consisting of $\bk_f$-finite functions $\phi_f: G(\ba_f)\rar \tau_f$ that satisfy
\[
\phi_f\left(\smalltwomatrix{A}{\ast}{}{x\leftup{t}{A^{-1}}}g_f\right)=\big|\frac{\det A}{x}\big|^2\tau_f(A)(\phi_f(g_f)).
\]

We define a map from the RHS of (\ref{cohop_global}) to the LHS.



(i) Let $v_\pm$ be as in (\ref{cohop}). To $\phi_f\in \Pi(\tau_f,\frac{1}{2})$, we associate two functions $F_\pm(g)$ on $P(\br)\times G(\ba_f)$: at $g_f\in G(\ba_f)$ and $p=\smalltwomatrix{A}{\ast}{}{x\leftup{t}{A^{-1}}}\in P(\ba)$, set
\[
F_\pm(pg_f)=\lambda^2(p)\cdot [v_\pm\otimes \phi_f(g_f)](A).
\]
Here $[v_\pm\otimes \phi_f(g_f)]$ refer to  the cuspidal automorphic forms in the space of $\tau$ that correspond to $v_\pm\otimes \phi_f(g_f)$ under the isomorphism $\tau=\fD_4\otimes \tau_f$. 

(ii) Accordingly, define a differential form $\omega$ on $P(\br)\times G(\ba_f)$:
\[
\omega:=F_+(g)\eta^-\wedge\eta_-+F_-(g)\eta^+\wedge\eta_+
\]
Because of (\ref{cohop}), $\omega$ is right invariant by $K_{M,\infty}$ and descends to a form on $P(\br)\times G(\ba_f)/K_\infty$. By the identification $P(\br)/K_{M,\infty}=G(\br)/K_\infty$, it is regarded as a form on $G(\ba)/K_\infty$. Its pullback to $G(\ba)$ is
\[
\omega(p_\infty k_\infty, g_f)=R_{k_\infty^{-1}}^\ast \big(\omega(p_\infty,g_f)\big),\quad p_\infty\in P(\br), k_\infty\in K_\infty, g_f\in G(\ba_f).
\]
By (\ref{cohop_global}), the form $\omega$ is closed and belongs to $\mathrm{H}^3(\fg,K_\infty,\Pi(\tau,\frac{1}{2}))$. 

\subsubsection{Eisenstein series operation}

Regard $\lambda$ as a function on $G(\ba)$ by setting $\lambda(pk)=\lambda(p)$ for $p\in P(\ba), k\in \bk$. Define $\eta_o:=\lambda^*(\frac{dt}{t})$. It is a closed degree-$1$ form on $P(\bq)\backslash G(\ba)/K_\infty$ and $\eta_o(p)=2a^\ast$ for $p\in P(\ba)$. 

To $\kappa\in C^\infty_c(\br_+)$ and $\omega\in \HH^3(\fg,K_\infty,\Pi(\tau,\frac{1}{2}))$, we associate a differential form $\eta^{\kappa,\omega}$ on $P(\bq)\backslash G(\ba)/K_\infty$,
\[
 \eta^{\kappa,\omega}:=\kappa(\lambda(g))\omega\wedge \eta_o.
\]
Because $\omega$ and $\eta_o$ are closed and $d\big(\kappa(\lambda(g))\big)=\lambda(g)\kappa^\prime(\lambda(g))\eta_0$,
there is $d(\eta^{\kappa,\omega})=d\big(\kappa(\lambda(g))\big)\wedge \omega\wedge \eta_o=0$, whence $\eta^{\kappa,\omega}$ is a closed form.

Imitating the construction of Eisenstein series, we define
\begin{align*}\label{ecoho_p}
&E(\eta^{\kappa,\omega})=\sum_{\gamma\in P(\bq)\backslash G(\bq)}\ L_{\gamma}^*\big(\eta^{\kappa,\omega}\big),\\
\end{align*}
It is a $\bk_f$-finite differential form on $G(\bq)\backslash G(\ba)/K_\infty$.
\begin{remark}\label{rm_eclassp}
Choose a basis $\{\omega_i\}$ of $\fp^\ast$ and write $\eta_o=\sum_i c_i(g)\omega_i$, then $c_i(g)$ are left $P(\ba)$-invariant and right $\bk_f$-invariant because $\eta_o$ is so. Write $\omega=\sum_I F_I(g)\omega_I$ with $F_I(g)\in \Pi(\tau,\frac{1}{2})$ and $\omega_I\in \wedge^3\fp^\ast$, then
\begin{equation}\label{eclassp}
E(\eta^{\kappa,\omega})=\sum_{i,I}E(\kappa\circ \lambda\cdot c_iF_I)\omega_I\wedge \omega_i.
\end{equation}
Here $c_iF_i\in \Pi(\tau,\frac{1}{2})$ and $E(\kappa\circ \lambda\cdot c_iF_I)=\sum_{\gamma\in P(\bq)\backslash G(\bq)}[\kappa\circ\lambda\cdot c_iF_I](\gamma g)$ are pseduo-Eisenstein series of type $(P,\tau)$. $E(\eta^{\kappa,\omega})$ is closed by Lemma ~\ref{closedness} and we may call it a pseudo-Eisenstein cohomological form.

\end{remark}

\begin{lemma}\label{closedness}
Suppose $\tau\subset \ma_{cusp}(\PGL_2)$ with $\tau_\infty=\fD_4$. $E(\eta^{\kappa,\omega})$ is a $\bk_f$-finite rapidly decreasing closed form on $G(\bq)\backslash G(\ba)/K_\infty$.
\end{lemma}

\begin{proof}
By Proposition II.1.10 in \cite{moegwald95}, the pseudo-Eisenstein series $E(\kappa\circ \lambda\cdot c_iF_I)$ in (\ref{eclassp}) are rapidly decreasing. So is $E^{\kappa,\omega}$.

For closedness, as observed by Harder in \cite{harder1975}, it suffices to show that for each $\bq$-parabolic subgroup $\mathrm{P}$ of $G$, the constant term of $E(\eta^{\kappa,\omega})$ along the unipotent radical of $\mathrm{P}$ is a closed form on $\mathrm{P}(\bq)\backslash G(\ba)/K_\infty$. One may suppose that $\mathrm{P}$ is one of $B, P$ and $Q$. The constant terms along $B$ and $Q$ are zero because of the cuspidality of $\tau$.

We now compute the constant term along $P$. Set
\[
 w_1=1,\
 w_2=\left(\begin{smallmatrix}
 & &1 &\\
 &1& &\\
 -1& & &\\
 & & &1
 \end{smallmatrix}\right),\
 w_3=\left(\begin{smallmatrix}
 1& & &\\
 && &1\\
 & &1 &\\
 &-1 & &
 \end{smallmatrix}\right),\
 w_4=\left(\begin{smallmatrix}
 & &1 &\\
 && &1\\
 1& & &\\
 &1 & &
 \end{smallmatrix}\right).
\]
There is $G(\bq)=\sqcup_{j=1}^4 P(\bq)w_j P(\bq)$. Set $U_{j}=w_j^{-1}Pw_j\cap U$ and $M_{j}=w_j^{-1}Pw_j\cap M$, then $P(\bq)\omega_j P(\bq)$ is the disjoint union of  $P(\bq)\omega_j \gamma_j \gamma_j^\prime$ with $\gamma_j\in U_j(\bq)\backslash U(\bq)$ and $\gamma_j^\prime\in M_j(\bq)\backslash M(\bq)$. So
\[
E_P(\eta^{\kappa,\omega})=\int_{[U]} L_u^*E(\eta^{\kappa,\omega})du=\sum_j \sum_{\gamma_j,\gamma_j^\prime} \int_{[U]} L_{\omega_j \gamma_j \gamma_j^\prime u}^*\eta^{\kappa,\omega}du.
\]
We make a change of variable $u\rar {\gamma_j^\prime}^{-1}u\gamma_j^\prime$ and then combine the summation over $\gamma_j$ with the integration over $[U]$. It yields
\begin{equation}\label{constanterm}
E_P(\eta^{\kappa,\omega})=\sum_j \sum_{\gamma_j^\prime}\int_{U_j(\bq)\backslash U_j(\ba)}L_{\omega_j u\gamma_j^\prime}^\ast \eta^{\kappa,\omega}.
\end{equation}
When $j=2,3$, there exists a $1$-dimensional subgroup $U_j^\prime\subset U_j$ satisfying $w_jU_j^\prime w_j^{-1}\subset M$, whence $\int_{U_j(\bq)\backslash U_j(\ba)}L_{\omega_j u\gamma_j^\prime}^\ast \eta^{\kappa,\omega}=0$ by the cuspidality of $\tau$. So only the terms for $j=1, 4$ remain in (\ref{constanterm}),
\[
 E_P(\eta^{\kappa,\omega}) =\eta^{\kappa,\omega}+\int_{U(\ba)}
 L_{w_4 u}^*\eta^{\kappa,\omega}du
\]

We now show $\eta^\prime:=\int_{U(\ba)} L_{w_4 u}^*\eta^{\kappa,\omega}du$ is closed. Write $\omega=\sum_I F_I \omega_I$ with $F_I\in \Pi(\tau,\frac{1}{2})$ and $\omega_I\in \wedge^3\fp^\ast$. Put $F_{I,\kappa}=F_I(g)\kappa\big(\lambda(g)\big)$, then
\begin{align*}
&\eta^{\kappa,\omega}=\sum_I F_{I,\kappa}(g)\omega_I\wedge \eta_0,\\
&\eta^\prime=\sum_I \omega_I\wedge \big[\int_{U(\ba)}L^\ast_{w_4u}(F_{I,\kappa}\eta_o) du\big].
\end{align*}
Observe that
\begin{align*}
d\big[\int_{U(\ba)}L^\ast_{w_4u}(F_{I,\kappa}\eta_o) du\big]=&\int_{U(\ba)}d\big[L^\ast_{w_4u}(F_{I,\kappa}\eta_o)\big]du.
\end{align*}
Here one can change the order of differentiation and integration because $\kappa$ is compactly supported and $F_I$ is cuspidal on $M(\ba)$. So $d\eta^\prime$ equals
\begin{align*}
&\sum_I d\omega_I\wedge \big[\int_{U(\ba)}L^\ast_{w_4u}(F_{I,\kappa}\eta_o) du\big]-\sum_I\omega_I\wedge d\big[\int_{U(\ba)}L^\ast_{w_4u}(F_{I,\kappa}\eta_o) du\big]\\
=&\sum_I \int_{U(\ba)} d\omega_I\wedge L^\ast_{w_4u}(F_{I,\kappa}\eta_o)- \omega_I\wedge d[L^\ast_{w_4u}(F_{I,\kappa}\eta_o)]\\
=&\sum_I \int_{U(\ba)} d\big[L^\ast_{w_4u} (\eta^{\kappa,\omega})\big]du.
\end{align*}
Because $\eta^{\kappa,\omega}$ is closed, there is $d\big[L^\ast_{w_4u} (\eta^{\kappa,\omega})\big]=L^\ast_{w_4u} (d\eta^{\kappa,\omega})=0$, whence $d\eta^\prime=0$ and $\eta^\prime$ is closed. Thus, $E(\eta^{\kappa,\omega})$ is closed.
\end{proof}

\subsection{Properties of $E(\eta^{\kappa,\omega})$}

\begin{lemma}\label{vp_p}
$<\mathrm{H}^{1,1}(\pi_f^\prime),E(\eta^{\kappa,\omega})>=0$ for $\pi_f^\prime\neq \pi_f$.
\end{lemma}
\begin{proof}
By Section ~\ref{repofh2}, nonzero $\HH^{1,1}(\pi_f^\prime)$ are of the form
\[
\HH^{1,1}(\pi_f^\prime)=\chi^\prime(\nu(g))\cdot \HH^{1,1}(\fg,K_\infty,\pi^\pprime),
\]
where $\chi^\prime$ is a character with $\chi^\prime_\infty\in \{1,\sgn\}$ and $\pi^\pprime$ is of type I, II or III. Hence a form $\omega^\prime$ in $\HH^{1,1}(\pi_f^\prime)$ is of the shape
\[
\omega^\prime=\chi^\prime(\nu(g))\sum_{J} \varphi_J \omega_J,\quad \varphi_J\in \pi^\pprime,\ \omega_J\in \wedge^2\fp^\ast.
\]

Write $E(\eta^{\kappa,\omega})=\sum_{i,I}E(\kappa\circ \lambda \cdot c_iF_I)\omega_I\wedge \omega_i$ as in (\ref{eclassp}) of Remark ~\ref{rm_eclassp}, with  $\omega_I\in \wedge^3\fp^\ast$, $\omega_i\in \fp^\ast$, $F_I(g)\in \Pi(\tau,\frac{1}{2})$, and $c_i(g)$ left invariant by $P(\ba)$. It follows that
\[
<\omega^\prime,E(\eta^{\kappa,\omega})>=\sum_{i,I,J} \int_M \chi^\prime(\nu(g))\varphi_J(g)E(\kappa\circ \lambda \cdot F_I)\omega_J\wedge \omega_I\wedge \omega_i.
\]
The form $\omega_J\wedge \omega_I\wedge \omega_i$ is either zero or a scalar multiple of the volume form on $M$. So $<\omega^\prime,E(\eta^{\kappa,\omega})>$ is a finite sum of integrals of the following type,
\begin{align}\label{pi_g}
\nonumber &\int_{G(\bq)Z(\br)^+\backslash G(\ba)} \chi^\prime(\nu(g))\varphi(g) E(\kappa\circ \lambda \cdot F) dg\\
=&\int_{P(\bq)U(\ba)Z(\br)^+\backslash G(\ba)}\lambda(p)^{-3}\chi^\prime(\nu(g))\varphi_P(g) \kappa\big(\lambda(g)\big) F(g)dg,
\end{align}
where $\varphi\in \pi^\pprime$, $F\in \Pi(\tau,\frac{1}{2})$ and $\varphi_P(g):=\int_{[U]}\varphi(ug)du$ refers to the constant term of $\varphi$ along $P$. (Note that $c_i(g)F_I(g)\in \Pi(\tau,\frac{1}{2})$.)

(i) When $\pi^\pprime$ is of type I, $\varphi_P=0$ and the above integral vanishes. When $\pi^\pprime$ is of type III, $\varphi$ is constant and the above integral vanishes because $F$ is cuspidal when restricted to $M(\ba)$. So $<\omega^\prime,E(\eta^{\kappa,\omega})>=0$.

(ii) $\pi^\pprime$ is of type II.  So $\pi^\pprime=J(P,\tau^\pprime,\frac{1}{2})$ with $\tau^\pprime\subset \ma_{cusp}(\PGL_2)$ and $L(\frac{1}{2},\tau^\pprime)\neq 0$. For $\varphi\in \pi^\pprime$, the constant term $\varphi_P$ belongs to the space $\Pi(\tau^\pprime,-\frac{1}{2})$. Writing elements in $G(\ba)$ as $g=\smalltwomatrix{A}{}{}{x\leftup{t}{A^{-1}}}k$ with $k\in \bk$, we can rewrite the integral in (\ref{pi_g}) as
\begin{equation}\label{pp_pintegral}
\int_{\bk}\int_{M(\bq)Z(\br)^+\backslash M(\ba)}\chi^\prime(x)\kappa(\big|\frac{\det A}{x}\big|)f^\pprime_k(A)f_k(A)d^\times A d^\times x,
\end{equation}
with
\begin{align*}
&f_k^\pprime(A)=|\det A|^{-1}\varphi_P\big(\smalltwomatrix{A}{}{}{\leftup{t}{A^{-1}}}k\big)\in \tau^\pprime,\\
&f_k(A)=|\det A|^{-2}F\big(\smalltwomatrix{A}{}{}{\leftup{t}{A^{-1}}}k\big)\in \tau.
\end{align*}

Setting $x=x\det A$, we can turn the inner integral in (\ref{pp_pintegral}) into
\begin{equation*}\label{sintegral}
\int_{\bq^\times\backslash \ba^\times}\chi^\prime(x)\kappa(|x^\prime|^{-1})d^\times x\cdot \int_{\bq^\times\br_+\backslash \ba^\times}\chi^\prime(z)d^\times z \int_{[\PGL_2]}f^\pprime_k(A)f_k(A)d^\times A.
\end{equation*}
For the above expression to be nonzero, it is necessary that $\chi^\prime=1$ and $\tau^\pprime={\tau}^\vee=\tau$. Thus, for (\ref{pp_pintegral}) and (\ref{pi_g}) to be nonzero, it is necessary that $\pi^\prime=\pi$. Therefore, $<\omega^\prime,E(\eta^{\kappa,\omega})>=0$ for $\pi^\prime_f\neq \pi_f$.
\end{proof}

\begin{lemma}\label{nvp_p}
Suppose $\tau\in \ma_{cusp}(\PGL_2)$ with $\tau_\infty=\fD_4$ and $L(\frac{1}{2},\tau)\neq 0$. There exists $\kappa\in C^\infty_c(\br_+)$, $\omega\in \HH^3(\fg,K_\infty,\Pi(\tau,\frac{1}{2}))$ and a sufficiently small $K_f$ such that $\int_{Z_{H,K_f}}E(\eta^{\kappa,\omega})\neq 0$.
\end{lemma}
\begin{proof}
The map $\phi_f\in \Pi(\tau_f,\frac{1}{2})\rar \omega\in \HH^3(\fg,K_\infty,\Pi(\tau,\frac{1}{2}))$ in Section ~\ref{gicoho} guides the choice of $\omega$. To a decomposable vector $v_f=\otimes v_p\in\tau_f$, we associate a specific section $\phi_{v_f}=\otimes \phi_p\in \Pi(\tau_f,\frac{1}{2})$: let $S(v_f)$ be the set of finite places of $\bq$ such that $v_p$ is spherical; for $p\in S(v_f)$, choose the $\phi_p$ with $\phi_p|_{\bk_p}=v_p$; for $p\not\in S(v_f)$, set $\overline{U}=J_2^{-1}UJ_2$ and let $\phi_p$ be the one supported in $P(\bq_p)\overline{U}(\bq_p)$ with
\[
\phi_p\smalltwomatrix{I_2}{}{n}{I_2}=
\begin{cases}
v_p, &n\in \mathrm{Sym}_{2\times 2}(\bz_p),\\
0, &n\not\in \mathrm{Sym}_{2\times 2}(\bz_p). 
\end{cases}
\]

(i) Choose a decomposable $v_f\in \tau_f$, set $\phi_f=\phi_{v_f}$ and let $\omega$ be the form associated to $\phi_f$. Choose $K_f$ sufficiently small so that $\phi_f$ is $K_f$-invariant.  
Because $H(\bq)$ acts transitively on $P(\bq)\backslash G(\bq)$, there is
\begin{align}\label{pip_eq1}
\nonumber \int_{Z_{H,K_f}}E(\eta^{\kappa,\omega})=&\int_{H(\bq)\backslash H(\ba)/K_{H,\infty}}\sum_{\gamma\in P(\bq)\backslash G(\bq)}L_{\gamma}^\ast \eta^{\kappa,\omega} dg_f\\
=&\int_{P_H(\bq)\backslash H(\ba)/K_{H,\infty}} \eta^{\kappa,\omega} dh_f.
\end{align}
Note that $H(\br)/K_{H,\infty}=P_H(\br)/(K_{H,\infty}\cap P_H(\br))$ and that $\eta^{\kappa,\omega}$ is right-invariant by $K_{H,\infty}\cap P_H(\br)=Z(\br)^+\cdot \{I_4, \diag(1,-1,1,-1)\}$. So
\begin{align}\label{pip_eq2}
\int_{P_H(\bq)\backslash H(\ba)/K_{H,\infty}} \eta^{\kappa,\omega} dg_f
\nonumber =&\int_{P_H(\bq)\backslash P_H(\br)\times H(\ba_f)/ (K_{H,\infty}\cap P_H(\br)} \eta^{\kappa,\omega}dh_f\\
=&2\int_{P_H(\bq)Z(\br)^+ \backslash P_H(\br)\times H(\ba_f)} \eta^{\kappa,\omega}dh_f
\end{align}

(ii) Recall that $\eta^{\kappa,\omega}=\kappa\big(\lambda(g)\big)\omega\wedge \eta_o$. Also recall from Section ~\ref{gicoho} that $\omega=F_+(g)\eta^-\wedge\eta_-+F_-(g)\eta^+\wedge\eta_+$ with
\[
F_\pm(pg_f)=\lambda^2(p)\cdot [v_\pm\otimes \phi_f(g_f)](A),\quad p=\smalltwomatrix{A}{\ast}{}{x\leftup{t}{A^{-1}}},\ g_f\in G(\ba_f).
\]
The right translation by $\ \diag(1,-1,1,-1)$ sends $F_+(g)\eta^-\wedge\eta_-$ to $F_-(g)\eta^+\wedge\eta_+$, and vice versa. So the RHS of (\ref{pip_eq2}) is equal to
\begin{equation}\label{pip_eq3}
4\int_{P_H(\bq)Z(\br)^+ \backslash P_H(\br)\times H(\ba_f)} \kappa\big(\lambda(h)\big)F_+(h)\eta^-\wedge\eta_-\wedge \eta_o\cdot dh_f
\end{equation}

The form $(\eta^-\wedge\eta_-\wedge \eta_o)|_{P_H(\br)}=2ia^\ast\wedge h^\ast\wedge n_1^\ast \wedge n_2^\ast$ represents a left Haar measure on $P_H(\br)$. Write it as $c_{H,\infty}dp^\prime_\infty$, then (\ref{pip_eq3}) is equal to
\begin{equation}\label{pip_eq4}
4c_{H,\infty}\int_{P_H(\bq)Z(\br)^+ \backslash P_H(\br)\times H(\ba_f)} \kappa\big(\lambda(h)\big)F_+(p^\prime_\infty,h_f)dp^\prime_\infty dh_f.
\end{equation}

(iii) Set $\overline{U}_H:=\overline{U}\cap H$ and $H_{v_f}:=P_H(\ba_f)\cdot H^\prime_{v_f}$, with
\[
H^\prime_{v_f}=(\prod_{p\not\in S(v_f)} \overline{U}_H(\bz_p))\times (\prod_{p\in S(v_f)} \bk_{H,p}).
\]
By the choice of $\phi_f$, the function $F_+(p_\infty^\prime,h_f)$ is supported in $P_H(\br)\times H_{v_f}$ and is right invariant by $H_{v_f}^\prime$. So (\ref{pip_eq4}) is equal to
\begin{align}\label{pip_eq5}
\nonumber &4c_{H,\infty}\mathrm{Vol}(H^\prime_{v_f})\int_{P_H(\bq)Z(\br)^+ \backslash P_H(\ba)} \kappa\big(\lambda(p^\prime)\big)F_+(p^\prime)dp^\prime\\
=&4c_{H,\infty}\mathrm{Vol}(H^\prime_{v_f})\underset{M_H(\bq)Z(\br)^+ \backslash M_H(\ba)}{\int} \lambda(m^\prime)^{-2} \kappa\big(\lambda(m^\prime)\big)F_+(m^\prime)dm^\prime
\end{align}

Write $m^\prime=\diag(a_1,a_2,a_0a_1^{-1},a_0a_2^{-1})$ and $f_+=[v_+\otimes v_f]\in \tau$, then
\[
\kappa\big(\lambda(m^\prime)\big)=\kappa(\big|\frac{a_1a_2}{a_0}\big|),\quad F_+(m^\prime)=\lambda(m^\prime)^2 f_+\smalltwomatrix{a_1}{}{}{a_2}.
\]
By changing variables, one can simplify (\ref{pip_eq5}) and turn it into
\[
4c_{H,\infty}\mathrm{Vol}(H^\prime_{v_f})\mathrm{Vol}(\bq^\times \br_+\backslash \ba^\times)^2\int_{\br_+}\kappa(t)d^\times t\int_{\bq^\times\backslash \ba^\times}f_+\smalltwomatrix{a_1}{}{}{1}d^\times a_1.
\]
Thus, setting $C=4c_{H,\infty}\mathrm{Vol}(H^\prime_{v_f})c^2$, we have
\[
\int_{Z_{H,K_f}}E(\eta^{\kappa,\omega})=C\int_{\br_+}\kappa(t)d^\times t\cdot \int_{\bq^\times\backslash \ba^\times} f_+\smalltwomatrix{a_1}{}{}{1}d^\times a_1.
\]

(iv) When $L(\frac{1}{2},\tau)\neq 0$, the integral $\int_{\bq^\times\backslash \ba^\times} f_+\smalltwomatrix{a_1}{}{}{1}d^\times a_1$ is nonvanishing on $v_+\otimes \tau_f$. (This is well-known and follows from the Jacquet-Langlands theory \cite{jl70} for $L$-functions of $\GL(2)$-representations.) Choose $v_f$ such that this integral is nonzero and choose $\kappa$ such that $\int_{\br_+}\kappa(t)d^\times t\neq 0$, then the period $\int_{Z_{H,K_f}}E(\eta^{\kappa,\omega})$ is nonzero.
\end{proof}

\section{Nonvanishing Periods III}\label{np3}

We verify Proposition ~\ref{formOmega} for $\pi=1$. As in Section ~\ref{np2}, we consider the split divisor $Z_{H,K_f}$ associated to the group $H\subset G$ and construct a form $\Omega$ by using Eisenstein cohomology. Note that $H^2(\fg,K_\infty,1)=\bc \omega_0$ with
\[
\omega_0=h^*\wedge n_2^*+\frac{1}{2}n_0^*\wedge n_3^*+a^*\wedge n_1^*.
\]
Note that $\omega_0\in \wedge^2 \fb_0^\ast\cong \wedge^2 \fp^\ast$ is $K_\infty$-invariant.

\begin{proposition}\label{charactercase}
There exists a $\bk_f$-finite rapidly decreasing closed form $\Omega$ on $M$ such that \emph{(i)} $<\mathrm{H}^{1,1}(\pi_f^\prime),\Omega>=0$ for $\pi_f^\prime\neq 1$, \emph{(ii)} $\int_{Z_{H,K_f}} \Omega \neq 0$.
\end{proposition}

We prove Proposition ~\ref{charactercase} at the end of this section. Here is an immediate corollary.

\begin{corollary}
Let $\chi$ be a character of $\bq^\times\br_+\backslash \ba^\times$ and $\pi=\chi\circ \nu$, then $\mathrm{H}^{1,1}(\pi_f)$ is spanned by the image of certain split special divisor.
\end{corollary}

\subsection{Eisenstein cohomology}

Let $N$ be the unipotent radical of $B$ and $A$ be the diagonal subgroup of $B$. Set $K_{A,\infty}=K_\infty\cap A(\br)$.

To $\tau\in C^\infty_c(\br_+\times \br_+)$, we associate a function $f^\tau$ on $B(\bq)\backslash G(\ba)/K_\infty$:
\[
f^\tau(g)=\tau(|\frac{a_1}{a_2}|,|\frac{a_2^2}{a_0}|),
\]
for $g=nak$ with $n\in N(\ba)$, $a=\diag(a_1,a_2,\frac{a_0}{a_1},\frac{a_0}{a_2})\in A(\ba)$, and $k\in \bk$. 

To $\tau_1, \tau_2 \in C^\infty_c(\br_+\times \br_+)$, we associate a differential form $\eta^{\tau_1,\tau_2}$ on $B(\bq)\backslash B(\br)\times G(\ba_f)/K_{A,\infty}$:
\[
\eta^{\tau_1,\tau_2}(g)=f^{\tau_1}(g) a^* \wedge h^*\wedge n_1^* \wedge n_2^*+
 f^{\tau_2} (g) a^* \wedge n_0^* \wedge (n_1^*-n_2^*) \wedge n_3^*.
\]
Because $B(\br)/K_{A,\infty}=G(\br)/K_\infty$, the form $\eta^{\tau_1,\tau_2}$ can be regarded as a differential form on $B(\bq)\backslash B(\br)\times G(\ba_f)/K_\infty$. Its pullback to $G(\ba)$ is
\[
\eta^{\tau_1,\tau_2}(b_\infty k_\infty,g_f)=R_{k_\infty^{-1}}\big(\eta^{\tau_1,\tau_2}(b_\infty,g_f)\big),
\]
for $b_\infty\in B(\br)$, $k_\infty\in K_\infty$, and $g_f\in G(\ba_f)$.

Imitating the construction of Eisenstein series, we define
\begin{align*}
E(\eta^{\tau_1,\tau_2})=\sum_{\gamma\in B(\bq)\backslash G(\bq)} L_\gamma^*\big(\eta^{\tau_1,\tau_2}\big).
\end{align*}
\begin{remark}\label{exp_eb}
$\eta^1:=a^* \wedge h^*\wedge n_1^* \wedge n_2^*$ and $\eta^2:=h^* \wedge n_0^* \wedge (n_1^*-n_2^*) \wedge n_3^*$ are left-invariant differential forms on $B(\bq)\backslash B(\br)\times G(\ba_f)/K_{A,\infty}$. We use the same notations to denote their pullback to $G(\ba)$. Choose a basis $\{\omega_I\}$ of $\wedge^4\fp^\ast$ and write $\eta^i=\sum_I C_{i,I}(g)\omega_I$, then $C_{i,I}(g)$ are left $B(\ba)$-invariant and right $\bk_f$-invariant because $\eta^i$ are so. There is
\begin{equation}\label{eclass4}
E(\eta^{\tau_1,\tau_2})=\sum_{i,I}E(C_{i,I}f^{\tau_i})\omega_I.
\end{equation}
Here $E(C_{i,I}f^{\tau_i}):=\sum_{\gamma\in B(\bq)\backslash G(\ba)}C_{i,I}(\gamma g)f^{\tau_i}(\gamma g)$ are Borel-type pseudo-Eisenstein series and are rapidly decreasing. So $E(\eta^{\tau_1,\tau_2})$ is rapidly decreasing.
\end{remark}

\begin{lemma}\label{closedness1}
$\eta^{\tau_1,\tau_2}$ is closed if and only if $\tau_1=2\tau_2-2t_1\ppder{\tau_2}{t_1}+2t_2\ppder{\tau_2}{t_2}$.
\end{lemma}
\begin{proof}
One needs to calculate $d(\eta^{\tau_1,\tau_2})$. Recall that for a differential form $\Omega$ of degree $m$, $d\Omega$ is defined by the expression
\begin{align*}
 d\Omega(X_0,\cdots,&X_m)=\sum_{i=0}^m (-1)^i X_i
\left(\Omega(X_0,\cdots,\hat{X_i},\cdots,X_m)\right)\\
&+\sum_{0\leq i< j\leq m}
(-1)^{i+j}\Omega([X_i,X_j],X_0,\cdots,\hat{X_i},\cdots,\hat{X_j},\cdots,X_m),
\end{align*}
where $X_0,\cdots,X_m$ are smooth vector fields. 

$\{a^*,h^*,n_0^*,n_1^*,n_2^*,n_3^*\}$ is a frame of the cotangent bundle on $G(\br)/K_\infty=B(\br)/K_{A,\infty}$. Direct calculation shows that
\begin{align*}
&d(a^*)=d(h^*)=0,\quad d(n_0^\ast)=-2h^\ast\wedge n_0^\ast,\\
&d(n_1^\ast)=-2a^\ast \wedge n_1^\ast-2h^\ast\wedge n_2^\ast-n_0^\ast\wedge n_3^\ast,\\
&d(n_2^\ast)=-2a^\ast \wedge n_2^\ast-2h^\ast\wedge n_1^\ast-n_0^\ast\wedge n_3^\ast,\\
&d(n_3^\ast)=-2a^\ast\wedge n_3^\ast-n_0^\ast\wedge (n_1^\ast+n_2^\ast),\\
&d(f^\tau)=f^{2t_2\ppder{\tau}{t_2}}a^* +f^{2t_1\ppder{\tau}{t_1}-2t_2\ppder{\tau}{t_2}}h^*.
\end{align*}
It follows that
\begin{align*}
&d(n_1^*\wedge n_2^*)=-4a^*\wedge n_1^*\wedge n_2^*
-n_0^*\wedge
(n_1^*-n_2^*)\wedge n_3^*;\\
&d(n_0^*\wedge (n_1^*-n_2^*)\wedge n_3^*)=-4a^*\wedge n_0^*\wedge
(n_1^*-n_2^*)\wedge n_3^*.
\end{align*}
Hence
\begin{align*}
 d(\eta^{\tau_1,\tau_2})
&=(-f^{\tau_1}-f^{2t_1\ppder{\tau_2}{t_1}-2t_2\ppder{\tau_2}{t_2}}
+2f^{\tau_2})a^*\wedge h^*\wedge n_0^*\wedge (n_1^*-n_2^*)\wedge n_3^*\\
&=f^{-\tau_1-2t_1\ppder{\tau_2}{t_1}+2t_2\ppder{\tau_2}{t_2}+2\tau_2}
a^*\wedge h^*\wedge n_0^*\wedge (n_1^*-n_2^*)\wedge n_3^*.
\end{align*}
Therefore, $d(\eta^{\tau_1,\tau_2})=0$ if and only $\tau_1=2\tau_2-2t_1\ppder{\tau_2}{t_1}+2t_2\ppder{\tau_2}{t_2}$.
\end{proof}

\begin{lemma}
$E(\eta^{\tau_1,\tau_2})$ is closed when $\tau_1=2\tau_2-2t_1\ppder{\tau_2}{t_1}+2t_2\ppder{\tau_2}{t_2}$.
\end{lemma}
\begin{proof}
Express $E(\eta^{\tau_1,\tau_2})$ using equation (\ref{eclass4}). We first use reduction theory to show that the summation defining $E(C_{i,I}f^{\tau_i})$ is locally finite. Write $G(\ba)^1=\{g\in G(\ba):|\nu(g)|=1\}$, then $G(\ba)=Z(\br)^+G(\ba)^1$. For $c>0$, Let $\fS(c)$ be the set of $g\in G(\ba)^1$ of the form $nak$ with $n\in N(\ba)$, $a\in A(\ba)$, $k\in \bk$ and $|\frac{a_1}{a_2}|$, $|\frac{a_2^2}{a_0}|\geq c$, then
\begin{itemize}
\item[(i)] $G(\ba)^1=G(\bq)\fS(c)$ when $c$ is small enough. 

\item[(ii)] the set $\{\gamma\in G(\bq):\gamma\fS(c)\cap \fS(c^\prime)\neq \emptyset\}$ for given $c,c^\prime$ is a finite union of cosets in $B(\bq)\backslash G(\bq)$.
\end{itemize}

Given $\tau_1, \tau_2$, there exists $c^\prime$ such that $f^{\tau_i}(g)=0$ ($i=1,2$) when $g\not\in Z(\br)^+\fS(c^\prime)$.  Choose a small $c$ so that $G(\ba)^1=G(\bq)\fS(c)$. By (ii), there are only finitely many $[\gamma_j]\in B(\bq)\backslash G(\bq)$ such that $\gamma_j\fS(c)\cap \fS(c^\prime)\neq \emptyset$, whence $E(C_{i,I}f^{\tau_i})(g)=\sum_{\gamma_j}C_{i,I}(\gamma g)f^{\tau_i}(\gamma_i g)$ is a finite summation for $g\in \fS(c)$. It follows that $E(\eta^{\tau_1,\tau_2})=\sum_j  L_{\gamma_i}^*(\eta^{\tau_1,\tau_2})$ is a finite summation for $g\in \fS(c)$.

When $\tau_1=2\tau_2-2t_1\ppder{\tau_2}{t_1}+2t_2\ppder{\tau_2}{t_2}$, Lemma ~\ref{closedness1} tells that $\eta^{\tau_1,\tau_2}$ is a closed form on $G(\ba)/K_\infty$, whence each $L_{\gamma_j}^*(\eta^{\tau_1,\tau_2})$ is closed. So $E(\eta^{\tau_1,\tau_2})$ is a closed form on $\fS(c)^\circ$, the interior of $\fS(c)$. Since $G(\ba)^1=G(\bq)\fS(c)^\circ$ when $c$ is small enough, $E(\eta^{\tau_1,\tau_2})$ is closed on $G(\ba)/K_\infty$.
\end{proof}

\subsection{Properties of $E(\eta^{\tau_1,\tau_2})$}

We suppose $\tau_1=2\tau_2-2t_1\ppder{\tau_2}{t_1}+2t_2\ppder{\tau_2}{t_2}$.
\begin{lemma}
$<\mathrm{H}^{1,1}(\pi_f^\prime),E(\eta^{\tau_1,\tau_2})>=0$ when $\pi_f^\prime\neq 1$.
\end{lemma}
\begin{proof}
By Section ~\ref{repofh2}, a nonzero $\HH^{1,1}(\pi_f^\prime)$ is of the form
\[
\HH^{1,1}(\pi_f^\prime)=\chi^\prime(\nu(g))\cdot \HH^{1,1}(\fg,K_\infty,\pi^\pprime),
\]
where $\chi^\prime$ is a character with $\chi^\prime_\infty\in \{1,\sgn\}$ and $\pi^\pprime$ is of type I, II or III.

(i) If $\pi^\pprime=1$ is of type III, then $\HH^{1,1}(\pi_f^\prime)=\bc\cdot \chi^\prime(\nu(g))\omega_0$. Since $\pi_f^\prime\neq 1$, there is $\chi^\prime_f\neq 1$. By unfolding the integral, we have
\begin{align*}
<\chi^\prime(\nu(g))\omega_0, E(\eta^{\tau_1,\tau_2})>=\int_{B(\bq)\backslash G(\ba)/K_{A,\infty}} \chi^\prime(\nu(g))\big(\frac{1}{2}f^{\tau_1}(g)+f^{\tau_2}(g)\big)\omega_\fp,
\end{align*}
where $\omega_\fp:=a^\ast\wedge h^\ast\wedge n_0^\ast \wedge n_1^\ast \wedge n_2^\ast \wedge n_3^\ast$ is the volume form on $B(\br)/K_{A,\infty}$. Noticing $|K_{A,\infty}/Z(\br)^+|=4$, we rewrite the RHS as
\begin{equation}\label{pi_beq1}
4 \int_{B(\bq)Z(\br)^+\backslash G(\ba)} \chi^\prime(\nu(g))\big(\frac{1}{2}f^{\tau_1}(g)+f^{\tau_2}(g)\big)dg
\end{equation}
Because $\chi_f^\prime\neq 1$ and $\chi_\infty^\prime\in \{1,\sgn\}$, there is $\chi^\prime|_{\ba^\times_1}\neq 1$. However, $f^{\tau_i}$ are left-invariant under $\diag(a_1,a_2,\frac{a_0}{a_1},\frac{a_0}{a_2})$ when $a_i\in \ba^\times_1$. Hence, (\ref{pi_beq1}) vanishes because $\int_{\bq^\times\backslash \ba_1^\times} \chi^\prime(z)d^\times z=0$. So $<\HH^{1,1}(\pi_f^\prime),E(\eta^{\tau_1,\tau_2})>=0$.


(ii) If $\pi^\pprime$ is of type I or II, then $\pi^\pprime_\infty=\pi^{2+}$. By the description of $\HH^{1,1}(\fg,K_\infty,\pi^{2+})$ in Lemma ~\ref{cohopi2}, 
forms in $\HH^{1,1}(\pi^\prime_f)$ are of the shape
\begin{equation}\label{omega2}
\omega^\prime=\chi^\prime(\nu(g))\sum_{j=-2}^2 \varphi_j\eta_{-j},\quad \varphi_j\in \pi^\pprime,\,  \eta_j\in \wedge^2\fp^\ast.
\end{equation}
Recalling the expression for $E(\eta^{\tau_1,\tau_2})$ in (\ref{eclass4}), we have
\[
\omega^\prime\wedge E(\eta^{\tau_1,\tau_2})=\sum_{i,j,I} \chi^\prime(\nu(g))\varphi_j(g)E(C_{i,I}f^{\tau_i})\eta_{-j}\wedge \omega_I,
\]
The form $\eta_{-j}\wedge \omega_I$, being left-invariant and of degree $6$, is either zero or a scalar multiple of the volume form $\omega_\fp$ on $M$. Thus, $<\omega^\prime,E(\eta^{\tau_1,\tau_2})>=\int_M \omega^\prime\wedge E(\eta^{\tau_1,\tau_2})$ is a finite sum of integrals of the following type: $\varphi\in \pi^\pprime$,
\begin{align*}
&\int_{G(\bq)Z(\br)^+\backslash G(\ba)} \chi^\prime(\nu(g))\varphi(g)E(C_{i,I}f^{\tau_i})(g)dg\\
=&\int_{B(\bq)N(\ba)Z(\br)^+\backslash G(\ba)}\chi^\prime(\nu(g))\varphi_B(g)C_{i,I}(g)f^{\tau_i}(g)dg.
\end{align*}
Here $\varphi_B(g):=\int_{N(\bq)\backslash N(\ba)}\varphi(ng)dn$ is the constant term of $\varphi$ along $B$. When $\pi^\pprime$ is of type I or II, $\varphi_B$ is zero and the above integral vanishes. Therefore, $<\mathrm{H}^{1,1}(\pi_f^\prime),E(\eta^{\tau_1,\tau_2})>=0$.
\end{proof}

\begin{lemma}\label{pi_last}
Suppose $\tau_i(t_1,t_2)=0$ when $t_1\in (0,1)$ ($i=1,2$), then
\[ \int_{Z_{H,K_f}} E(\eta^{\tau_1,\tau_2})=8c^3\underset{\br_+\times \br_+}{\int}\frac{\tau_1(t_1,t_2)}{|t_1t_2|^3}dt_1dt_2.
\]
\end{lemma}
\begin{proof}
Because $H(\bq)$ acts transitively on $P(\bq)\backslash G(\bq)$, there is
\begin{align}\label{pi_ebeq}
\nonumber \int_{Z_{H,K_f}} E(\eta^{\tau_1,\tau_2})=&\int_{P_H(\bq)\backslash H(\ba)/K_{H,\infty}} \sum_{\gamma\in B(\bq)\backslash P(\bq)} L_{\gamma}^*[\eta^{\tau_1,\tau_2}]dh_f\\
=&4\int_{P_H(\bq)Z(\br)^+\backslash P_H(\ba)} \sum_{\gamma\in B(\bq)\backslash P(\bq)} L_{\gamma}^*[\eta^{\tau_1,\tau_2}]dp_f^\prime.
\end{align}

(i) The coset space $B(\bq)\backslash P(\bq)$ is parameterized by $1$ and  $\gamma_\delta$ ($\delta\in \bq$), with $\gamma_\delta=\smalltwomatrix{\beta_\delta}{}{}{\leftup{t}{\beta_\delta}^{-1}}$ and $\beta_\delta=\smalltwomatrix{}{1}{-1}{}\smalltwomatrix{1}{\delta}{}{1}$. One needs to restrict
\[
L^\ast_{\gamma}[\eta^{\tau_1,\tau_2}]=f^{\tau_1}(\gamma g)L_{\gamma}^\ast \eta^1+f^{\tau_2}(\gamma g)L_{\gamma}^\ast \eta^2
\]
to $P_H(\ba)$, that is, to restrict $L^\ast_{\gamma} \eta^i$ to $P_H(\br)$. By definition,
\begin{equation}\label{transexp}
(L^\ast_{\gamma}\eta^i)|_{p_\infty^\prime}=L^\ast_{\gamma}\big(\eta^i(\gamma p^\prime_\infty)\big),\quad p^\prime_\infty\in P_H(\br).
\end{equation}
Write $\gamma p^\prime_\infty=p_\infty k_\infty$ with $p_\infty\in B(\br)$ and $k_\infty\in K_\infty$, then
\begin{equation}\label{etaexp}
\eta^i(p_\infty k_\infty)=R_{k_\infty^{-1}}^\ast \eta^i(p_\infty).
\end{equation}
For $\eta=\sum_{I}\omega_I$ with $\omega_I\in \wedge^4 \fp^\ast$ on $B(\br)$, there is $R_{k_\infty^{-1}}^\ast \eta=\sum_I \Ad_{k_\infty}^\ast \omega_I$.

\begin{itemize}
\item[(ia)] When $\gamma=1$, there are $\eta_1|_{P_H(\br)}=\eta_1$ and $\eta_2|_{P_H(\br)}=0$ because $n_0^\ast$ and $n_3^\ast$ restrict to zero on $P_H(\br)$.
\item[(ib)] Consider $\gamma=\gamma_\delta$. Set $k_\theta:=\smalltwomatrix{\cos \theta}{\sin \theta}{-\sin \theta}{\cos \theta}$ and $k(\theta):=\smalltwomatrix{k_\theta}{}{}{k_\theta}\in K_\br$ for $\theta\in \br$. For  $p^\prime_\infty=[\smalltwomatrix{r_1}{\ast}{}{r_0r_1^{-1}},\smalltwomatrix{r_2}{\ast}{}{r_0r_2^{-1}}]\in P_H(\br)\subset H(\br)$, 
there is $\gamma_\delta p^\prime_\infty\in B(\br)k(\theta)$ with $\theta=-\tan^{-1}\frac{r_1}{r_2\delta}$, whence
\[
(L^\ast_{\gamma_\delta}\eta^i)|_{p_\infty^\prime}=\Ad_{k(\theta)}^\ast\eta^i.
\]

\item[(ic)] The action of $\Ad_{k(\theta)}$ on $\fb_0\cong \fp$ is given by
\begin{align*}
&a\rar a, &n_0\rar n_0,\\
&h\rar (\cos 2\theta) h-(\sin 2\theta) n_0, &n_1\rar n_1,\\
&n_2\rar (\cos 2\theta)n_2-(\sin 2\theta)n_3, &n_3\rar -(\sin 2\theta)n_2+(\cos 2\theta)n_3.
\end{align*}
Hence $\Ad_{k(\theta)}^\ast$ acts on $\fb_0^\ast$ by
\begin{align*}
&a^\ast\rar a^\ast, &n_0^\ast\rar n_0^\ast-(\sin 2\theta)h^\ast,\\
&h^\ast\rar (\cos 2\theta) h^\ast, &n_1^\ast\rar n_1^\ast,\\
&n_2^\ast\rar (\cos 2\theta)n_2^\ast-(\sin 2\theta)n_3^\ast, &n_3\rar -(\sin 2\theta)n_2^\ast+(\cos 2\theta)n_3^\ast.
\end{align*}
Therefore $\Ad_{k(\theta)}\eta^1=(\cos 2\theta)a^\ast\wedge h^\ast\wedge n_1^\ast\wedge (\cos 2\theta\cdot n_2^\ast-\sin 2\theta\cdot n_3^\ast)$. Note that $n_3^\ast$ restricts to zero on $P_H(\br)\subset B(\br)$, so
\[
L^\ast_{\gamma_\delta}\eta^1|_{p_\infty^\prime}=(\cos 2\theta)^2 a^\ast\wedge h^\ast\wedge n_1^\ast\wedge n_2^\ast=\frac{(r_1^2-r_2^2\delta^2)^2}{(r_1^2+r_2^2\delta^2)^2}\cdot \eta^1
.
\]
Similarly, $L^\ast_{\gamma_\delta}\eta^2|_{p_\infty^\prime}=(\sin 2\theta)^2 a^\ast\wedge h^\ast\wedge n_1^\ast\wedge n_2^\ast=\frac{(2r_1r_2\delta)^2}{(r_1^2+r_2^2\delta^2)^2}\cdot \eta^1$.

To summarize, we have $L^\ast_{\gamma_\delta}\eta^i|_{p^\prime}=f_{i,\delta}(p^\prime)\eta^1$, where $f_{i,\delta}$ are functions on $P_H(\ba)$ given by
\[
f_{1,\delta}(p^\prime)=\frac{(r_1^2-r_2^2\delta^2)^2}{(r_1^2+r_2^2\delta^2)^2},\quad f_{2,\delta}(p^\prime)=\frac{(2r_1r_2\delta)^2}{(r_1^2+r_2^2\delta^2)^2},
\]
where $p^\prime\in P_H(\ba)$ with $p^\prime_\infty=[\smalltwomatrix{r_1}{\ast}{}{r_0r_1^{-1}},\smalltwomatrix{r_2}{\ast}{}{r_0r_2^{-1}}]$,
\end{itemize}

Noticing that $f_{1,0}(p^\prime)=1$ and $f_{2,0}(p^\prime)=0$, we have
\[
\sum_{\gamma\in B(\bq)\backslash P(\bq)} L_{\gamma}^*[\eta^{\tau_1,\tau_2}](p^\prime)=\big[2f^{\tau_1}+\sum_{\delta\in \bq^\times}f_{1,\delta}f^{\tau_1}(\gamma_\delta \cdot)+f_{2,\delta}f^{\tau_2}(\gamma_\delta \cdot)\big](p^\prime)\eta^1.
\]

(ii) Now we compute the integral in (\ref{pi_ebeq}). Note that $\eta^1$ is the volume form on $Z(\br)^+\backslash P_H(\br)$. By doing integration on  $U_H(\bq)\backslash U_H(\ba)$ first, we turn the integral in (\ref{pi_ebeq}) to
\begin{equation}\label{pi_ebeq3}
\int_{M_H(\bq)Z(\br)^+\backslash M_H(\ba)} \lambda(m^\prime)^{-2}\big[2f^{\tau_1}+\sum_{\delta\in \bq^\times}f_{1,\delta}f^{\tau_1}(\gamma_\delta \cdot)+f_{2,\delta}f^{\tau_2}(\gamma_\delta \cdot)\big](m^\prime)dm^\prime.
\end{equation}
We analyze the integrands $f^{\tau_1}$ and $f_{1,\delta}f^{\tau_1}(\gamma_\delta \cdot)+f_{2,\delta}f^{\tau_2}(\gamma_\delta \cdot)$ separately.

(iia) By changing variables, one computes quickly that
\[
\underset{M_H(\bq)Z(\br)^+\backslash M_H(\ba)}{\int} \frac{f^{\tau_1}(m^\prime)}{\lambda(m^\prime)^2}dm^\prime=c^3\underset{\br_+\times \br_+}{\int}\frac{\tau_1(t_1,t_2)}{|t_1t_2|^2}d^\times t_1d^\times t_2.
\]

(iib) Write $m^\prime=\diag(a_1,a_2,a_0a_1^{-1},a_0a_2^{-1})$ and then make a change of variables $a_1\rar a_1a_2$, $a_0\rar a_0a_2^2$. It yields
\begin{align}\label{2summand}
&\int_{M_H(\bq)Z(\br)^+\backslash M_H(\ba)}\lambda(m^\prime)^{-2}\big[\sum_{\delta\in \bq^\times}f_{1,\delta}f^{\tau_1}(\gamma_\delta \cdot)+f_{2,\delta}f^{\tau_2}(\gamma_\delta \cdot)\big](m^\prime)dm^\prime\\
\nonumber =&c\underset{(\bq^\times\backslash \ba^\times)^2}{\iint}\sum_{\delta\in\bq^\times, i}\big[f_{i,\delta}(\cdot)f^{\tau_i}(\gamma_\delta\cdot)\big](\diag(a_1,1,a_0a_1^{-1},a_0))|\frac{a_0}{a_1}|^2d^\times a_0 d^\times a_1.
\end{align}
The integrand of the double integral on the RHS can be written as
\begin{equation}\label{summation_last}
\sum_{\delta\in\bq^\times, i}\phi_i(a_1^{-1}\delta)f^{\tau_i}\smalltwomatrix{\beta(a_1,a_1^{-1}\delta)}{}{}{a_0\leftup{t}{\beta(a_1,a_1^{-1}\delta)^{-1}}}|\frac{a_0}{a_1}|^2,
\end{equation}
where $\beta(a_1,b_1):=\smalltwomatrix{1}{}{}{a_1}\smalltwomatrix{}{1}{-1}{}\smalltwomatrix{1}{b_1}{}{1}$ and $\phi_1, \phi_2$ are functions on $\ba^\times$ given by
\[
\phi_1(a_1)=\frac{(1-a_{1,\infty}^2)^2}{(1+a_{1,\infty}^2)^2},\quad \phi_2(a_1)=\frac{(2a_{1,\infty})^2}{(1+a_{1,\infty}^2)^2}.
\]
One actually has $\phi_i(a_1)=f_{i,\delta}(\diag(a_1,1,a_0a_1^{-1},a_0))$.

Combining the $\delta$-summation with the $a_1$-integration in (\ref{summation_last}), one rewrites the RHS of (\ref{2summand}) as
\begin{align*}
&c\iint_{a_1\in \ba^\times,\, a_0\in \bq^\times\backslash \ba^\times}\sum_{i=1,2} \phi_i(a_1^{-1})f^{\tau_i}\smalltwomatrix{\beta(a_1,a_1^{-1})}{}{}{a_0\leftup{t}{\beta(a_1,a_1^{-1})^{-1}}}d^\times a_0 d^\times a_1.
\end{align*}
Now we write $a_1=p^{-1}q\cdot (r,a_f)$ with $p,q\in \bz_{>0}$ being coprime and $r\in \br^\times, a_f\in \prod_{p<\infty}\bz_p^\times$. Accordingly,
\begin{equation}\label{integrand}
f^{\tau_i}\smalltwomatrix{\beta(a_1,a_1^{-1})}{}{}{a_0\leftup{t}{\beta(a_1,a_1^{-1})^{-1}}}=\tau_i\big(\frac{1}{p^2+r^2q^2},\frac{1}{|a_0|(p^2+r^2q^2)}\big).
\end{equation}
Because $\tau_i$ are compactly supported on $\br^+\times \br^+$, there are only finitely many $(p,q)$ such that the above expression is nonzero.

Specifically, if $\tau_i(t_1,t_2)=0$ when $t_1\in (0,1)$, then (\ref{integrand}) vanishes for all $(p,q)$ and hence for all $a_1\in \ba^\times$. As a consequence, the LHS of (\ref{2summand}) is zero and the period $\int_{Z_{H,K_f}}E(\eta^{\tau_1,\tau_2})$ is just $8$ times the expression in (iia). The lemma is proved.
\end{proof}

\begin{remark}
In the above computation, we only care about nonvanishing result and hence take $\mathrm{Vol}(K_f)=1$ for simplicity. In general, one needs to carefully choose the measure in order to get a  precise equality.
\end{remark}

\begin{proof}[Proof of Proposition ~\ref{charactercase}]
Set $\tau_i(t_1,t_2)=t_1^2t_2^2\wtilde{\tau}_i(t_1,t_2)$, then the condition $\tau_1=2\tau_2-2t_1\ppder{\tau_2}{t_1}+2t_2\ppder{\tau_2}{t_2}$ becomes $\wtilde{\tau}_1=2\wtilde{\tau}_2-2t_1\ppder{\wtilde{\tau}_2}{t_1}+2t_2\ppder{\wtilde{\tau}_2}{t_2}$. When $\wtilde{\tau}_i$ vanishes for $t_1\in (0,1)$, Lemma ~\ref{pi_last} tells that
\[
\int_{Z_{H,K_f}} E(\eta^{\tau_1,\tau_2})=8c^3\underset{\br_+\times \br_+}{\int}\frac{\tau_1(t_1,t_2)}{|t_1t_2|^3}dt_1d t_2=8c^3\underset{\br_+\times \br_+}{\int}\frac{\wtilde{\tau}_2(t_1,t_2)}{t_1t_2}dt_1d t_2.
\]
So we choose $\wtilde{\tau}\in C^\infty_c\big((1,\infty)^2\big)$ with ${\int}_{\br_+\times \br_+}\wtilde{\tau}(t_1,t_2)d^\times t_1d^\times t_2\neq 0$ and set $\tau_1=t_1^2t_2^2(2\wtilde{\tau}-2t_1\ppder{\wtilde{\tau}}{t_1}+2t_2\ppder{\wtilde{\tau}}{t_2})$, $\tau_2=t_1^2t_2^2\wtilde{\tau}$. The form $E(\eta^{\tau_1,\tau_2})$ then meets the requirement of Proposition ~\ref{charactercase}.
\end{proof}


\begin{thebibliography}{10}

\bibitem{SGA4.3}
{\em Th\'eorie des topos et cohomologie \'etale des sch\'emas. {T}ome 3}.
\newblock Lecture Notes in Mathematics, Vol. 305. Springer-Verlag, Berlin,
  1973.
\newblock S{\'e}minaire de G{\'e}om{\'e}trie Alg{\'e}brique du Bois-Marie
  1963--1964 (SGA 4), Dirig{\'e} par M. Artin, A. Grothendieck et J. L.
  Verdier. Avec la collaboration de P. Deligne et B. Saint-Donat.

\bibitem{bw2000}
A.~Borel and N.~Wallach.
\newblock {\em Continuous cohomology, discrete subgroups, and representations
  of reductive groups}, volume~67 of {\em Mathematical Surveys and Monographs}.
\newblock American Mathematical Society, Providence, RI, second edition, 2000.

\bibitem{borel1980}
Armand Borel.
\newblock Stable real cohomology of arithmetic groups. {II}.
\newblock In {\em Manifolds and Lie groups (Notre Dame, Ind., 1980)}, volume~14
  of {\em Progr. Math.}, pages 21--55. Birkh\"auser Boston, Mass., 1981.

\bibitem{Deligne71}
Pierre Deligne.
\newblock Travaux de {S}himura.
\newblock In {\em S\'eminaire Bourbaki, 23\`eme ann\'ee (1970/71), Exp. No.
  389}, pages 123--165. Lecture Notes in Math., Vol. 244. Springer, Berlin,
  1971.

\bibitem{wgan2008}
Wee~Teck Gan.
\newblock The {S}aito-{K}urokawa space of {${\rm PGSp}\sb 4$} and its transfer
  to inner forms.
\newblock In {\em Eisenstein series and applications}, volume 258 of {\em
  Progr. Math.}, pages 87--123. Birkh\"auser Boston, Boston, MA, 2008.

\bibitem{harder1975}
G.~Harder.
\newblock On the cohomology of discrete arithmetically defined groups.
\newblock In {\em Discrete subgroups of Lie groups and applications to moduli
  (Internat. Colloq., Bombay, 1973)}, pages 129--160. Oxford Univ. Press,
  Bombay, 1975.

\bibitem{hlr1986}
G.~Harder, R.~P. Langlands, and M.~Rapoport.
\newblock Algebraische {Z}yklen auf {H}ilbert-{B}lumenthal-{F}l\"achen.
\newblock {\em J. Reine Angew. Math.}, 366:53--120, 1986.

\bibitem{hh2012}
Hongyu He and Jerome~William Hoffman.
\newblock Picard groups of {S}iegel modular 3-folds and {$\theta$}-liftings.
\newblock {\em J. Lie Theory}, 22(3):769--801, 2012.

\bibitem{jl70}
H.~Jacquet and R.~P. Langlands.
\newblock {\em Automorphic forms on {${\rm GL}(2)$}}.
\newblock Lecture Notes in Mathematics, Vol. 114. Springer-Verlag, Berlin,
  1970.

\bibitem{kudla1997}
Stephen~S. Kudla.
\newblock Algebraic cycles on {S}himura varieties of orthogonal type.
\newblock {\em Duke Math. J.}, 86(1):39--78, 1997.

\bibitem{kudla_millson86}
Stephen~S. Kudla and John~J. Millson.
\newblock The theta correspondence and harmonic forms. {I}.
\newblock {\em Math. Ann.}, 274(3):353--378, 1986.

\bibitem{kudla_millson87}
Stephen~S. Kudla and John~J. Millson.
\newblock The theta correspondence and harmonic forms. {II}.
\newblock {\em Math. Ann.}, 277(2):267--314, 1987.

\bibitem{kudla_millson88}
Stephen~S. Kudla and John~J. Millson.
\newblock Tubes, cohomology with growth conditions and an application to the
  theta correspondence.
\newblock {\em Canad. J. Math.}, 40(1):1--37, 1988.

\bibitem{moegwald95}
C.~M{\oe}glin and J.-L. Waldspurger.
\newblock {\em Spectral decomposition and {E}isenstein series}, volume 113 of
  {\em Cambridge Tracts in Mathematics}.
\newblock Cambridge University Press, Cambridge, 1995.
\newblock Une paraphrase de l'\'Ecriture [A paraphrase of Scripture].

\bibitem{PS83}
I.~I. Piatetski-Shapiro.
\newblock On the {S}aito-{K}urokawa lifting.
\newblock {\em Invent. Math.}, 71(2):309--338, 1983.

\bibitem{Rama90}
Dinakar Ramakrishnan.
\newblock Problems arising from the {T}ate and {B}e\u\i linson conjectures in
  the context of {S}himura varieties.
\newblock In {\em Automorphic forms, Shimura varieties, and $L$-functions,
  Vol.\ II (Ann Arbor, MI, 1988)}, volume~11 of {\em Perspect. Math.}, pages
  227--252. Academic Press, Boston, MA, 1990.

\bibitem{VZ84}
David~A. Vogan, Jr. and Gregg~J. Zuckerman.
\newblock Unitary representations with nonzero cohomology.
\newblock {\em Compositio Math.}, 53(1):51--90, 1984.

\bibitem{Wald80}
J.-L. Waldspurger.
\newblock Correspondance de {S}himura.
\newblock {\em J. Math. Pures Appl. (9)}, 59(1):1--132, 1980.

\bibitem{Wald91}
Jean-Loup Waldspurger.
\newblock Correspondances de {S}himura et quaternions.
\newblock {\em Forum Math.}, 3(3):219--307, 1991.

\bibitem{Weiss92}
R.~Weissauer.
\newblock The {P}icard group of {S}iegel modular threefolds.
\newblock {\em J. Reine Angew. Math.}, 430:179--211, 1992.
\newblock With an erratum: ``Differential forms attached to subgroups of the
  Siegel modular group of degree two'' [J.\ Reine Angew.\ Math.\ {\bf 391}
  (1988), 100--156; MR0961166 (89i:32074)] by the author.

\bibitem{Weiss88}
Rainer Weissauer.
\newblock Differentialformen zu {U}ntergruppen der {S}iegelschen {M}odulgruppe
  zweiten {G}rades.
\newblock {\em J. Reine Angew. Math.}, 391:100--156, 1988.

\end{thebibliography}

\def\cprime{$'$} \def\cprime{$'$}

\end{document}